\documentclass[amscd,reqno,amssymb,verbatim,10pt]{amsart}
\usepackage[fontsize = 10bp]{fontsize}

\usepackage{amsthm}
\usepackage{amssymb, latexsym, amsmath, amscd, array, graphicx}
\usepackage{hyperref}
\usepackage[all]{xy}
\usepackage[scr]{rsfso}

\newtheorem{swtheorem}{Theorem}
\newtheorem{swcor}{Corollary}[swtheorem]
\swapnumbers
\usepackage{verbatim}
\usepackage{mathtools}
\usepackage{enumerate}
\usepackage{color}
\hypersetup{
    colorlinks=true,
    linkcolor=red,
    urlcolor=black,
    citecolor=blue
           }
\usepackage{tikz-cd}
\usepackage{bbm,graphicx}

\theoremstyle{plain}

\newtheorem{thm}{Theorem}[section]
\newtheorem{theorem}[thm]{Theorem}

\newtheorem{lemma}[thm]{Lemma}
\newtheorem{conjec}[thm]{Conjecture}
\newtheorem{prop}[thm]{Proposition}
\newtheorem{cor}[thm]{Corollary}

\usepackage{scalerel}
\usepackage{stackengine,wasysym}

\makeatletter
\def\@tocline#1#2#3#4#5#6#7{\relax
  \ifnum #1>\c@tocdepth 
  \else
    \par \addpenalty\@secpenalty\addvspace{#2}%
    \begingroup \hyphenpenalty\@M
    \@ifempty{#4}{%
      \@tempdima\csname r@tocindent\number#1\endcsname\relax
    }{%
      \@tempdima#4\relax
    }%
    \parindent\z@ \leftskip#3\relax \advance\leftskip\@tempdima\relax
    \rightskip\@pnumwidth plus4em \parfillskip-\@pnumwidth
    #5\leavevmode\hskip-\@tempdima
      \ifcase #1
       \or\or \hskip 1em \or \hskip 2em \else \hskip 3em \fi%
      #6\nobreak\relax
    \hfill\hbox to\@pnumwidth{\@tocpagenum{#7}}\par
    \nobreak
    \endgroup
  \fi}
\makeatother

\theoremstyle{definition}
\newtheorem{defn}[thm]{Definition}

\newtheorem{remark}[thm]{Remark}

\newtheorem{ex}[thm]{Example}

\newtheorem{question}[thm]{Question}

\DeclareMathOperator{\cat}{\mathsf{cat}}
\DeclareMathOperator{\dcat}{\mathsf{dcat}}

\DeclareMathOperator{\cu}{\mathsf{cl}}
\DeclareMathOperator{\zcl}{\mathsf{zcl}}
\DeclareMathOperator{\szcl}{\mathsf{szcl}}

\DeclareMathOperator{\TC}{\mathsf{TC}}
\DeclareMathOperator{\dTC}{\mathsf{dTC}}

\DeclareMathOperator{\cd}{{\rm cd}}

\DeclareMathOperator{\Ker}{{\rm Ker}}

\DeclareMathOperator{\supp}{{\rm supp}}

\def\pr{\protect\operatorname{pr}}

\def\B{{\mathcal B}}

\def\P{{\mathbb P}}
\def\T{{\mathbb T}}

\newcommand \pa[2]{\frac{\partial #1}{\partial #2}}

\def\scr{\mathcal}

\def\B{{\scr B}}

\def\C{{\mathbb C}}
\def\Z{{\mathbb Z}}
\def\Q{{\mathbb Q}}
\def\R{{\mathbb R}}

\def\1{\hbox{\rm\rlap {1}\hskip.03in{\rom I}}}
\def\Bbbone{{\rm1\mathchoice{\kern-0.25em}{\kern-0.25em}
{\kern-0.2em}{\kern-0.2em}I}}

\def\pa{\partial}

\def\wt{\widetilde}
\def\wh{\widehat}

\def\ov{\overline}

\long\def\forget#1\forgotten{}

\newcommand\ver[1]{\marginpar{\tiny Changed in Ver \VER}}

\date{\today}

\begin{document}

\title[Topological complexity of symmetric products]{LS-category and sequential topological complexity of symmetric products}

\author[Ekansh Jauhari]{Ekansh Jauhari}

\address{Ekansh Jauhari, Department of Mathematics, University
of Florida, 358 Little Hall, Gainesville, FL 32611-8105, USA.}
\email{ekanshjauhari@ufl.edu}


\begin{abstract}
The $n$-th symmetric product of a topological space $X$ is the orbit space of the natural action of the symmetric group $S_n$ on the product space $X^n$. 
In this paper, we compute the sequential topological complexities of (finite products of) symmetric products of closed orientable surfaces, thereby verifying the rationality conjecture of Farber and Oprea for these spaces. 
Additionally, we determine the Lusternik--Schnirelmann category of (finite products of) the symmetric products of closed non-orientable surfaces.
More generally, we provide lower bounds to the LS-category and the sequential topological complexities of the symmetric products of finite CW complexes $X$ in terms of the cohomology of $X$ and its products. 
On the way, we also obtain new lower bounds to the sequential distributional complexities of continuous maps and study the homotopy groups of the symmetric products of closed surfaces.
\end{abstract}




\maketitle
\tableofcontents

\section{Introduction}\label{Introduction}
The topological complexity of a topological space~\cite{Far1},~\cite{Ru} is a homotopy invariant that measures the complexity of moving points continuously within the space. More precisely, given a ``nice'' topological space $Y$ and an integer $m\ge 2$, the $m$-th sequential topological complexity of $Y$, denoted $\TC_m(Y)$, records the minimal \emph{order of instability}~\cite{Far2} of motion planning algorithms (on mechanical systems having configuration space $Y$) that take as input ordered $m$-tuples of positions in $Y$ and produce as output a uniformly timed motion of the system in $Y$ through those $m$ positions attained in that order. If $Y$ is the $n$-th symmetric product of a topological space $X$, denoted $SP^n(X)$, then positions in $Y$ become unordered $n$-tuples of positions in $X$. 

Though the notion of sequential topological complexity of spaces was initially motivated by topological robotics, it has been studied as a homotopy invariant for several classes of spaces, see, for example,~\cite{Far1},~\cite{Far2},~\cite{Ru},~\cite{BGRT},~\cite{LS}, \cite{FO},~\cite{AGO},~\cite{HL}, and the references therein. Furthermore, it is closely related to another homotopy invariant of spaces, called the Lusternik--Schnirelmann category ($\cat$), which is much older,~\cite{BG},~\cite{Ja},~\cite{CLOT}. These homotopy invariants are well-known to be special cases of Schwarz's general notion of the {sectional category} of fibrations,~\cite{Sch}. We refer the reader to~\cite{CLOT},~\cite{Far1},~\cite{Ru} for details.

In general, precisely determining the LS-category of a space is difficult, and determining (or even estimating) its (sequential) topological complexity can be even more difficult, see, for example,~\cite{FTY},~\cite{Dr2},~\cite{CV}. 

In this paper, we study the problem of estimating the sequential topological complexities and the Lusternik--Schnirelmann category of symmetric products of topological spaces, specifically finite CW complexes. The spaces of our particular interest are the symmetric products of closed $2$-manifolds (or surfaces), which are closed even-dimensional manifolds. 

Our first main result is the exact computation of the sequential topological complexities of the symmetric products of closed orientable surfaces $M_g$ of genus $g\ge 0$. The precise statement of our main result is as follows.

\begin{swtheorem}\label{aa}
For any integers $n,m \ge 2$ and $g\ge 0$, we have for the $m$-th sequential topological complexity of the $n$-th symmetric product of $M_g$ that
\[
\TC_m \left(SP^n(M_g)\right)=m\cat\left(SP^n(M_g)\right)=\begin{cases}
2mn & \text{ if } n\le g \\
mn+mg & \text{ if } n >g.
\end{cases}
\]
\end{swtheorem} 

To prove Theorem~\ref{aa}, we use the description of the rational cohomology ring of $SP^n(M_g)$ due to Macdonald~\cite{Mac}, and we show that for each $m\ge 2$, the cohomological lower bound to $\TC_m(SP^n(M_g))$ coincides with the LS-categorical upper bound, thereby determining $\TC_m(SP^n(M_g))$. See Section~\ref{lstcdefn} for some preliminary facts on the bounds to sequential topological complexities.

The sequential topological complexities of orientable surfaces were computed in \cite{Ru},~\cite{BGRT} and~\cite{LS}, and~\cite{GGGHMR}. The topological complexity ($\TC_2$) of symmetric products of orientable surfaces was computed very recently in~\cite{DDJ}. Theorem~\ref{aa} extends the work of~\cite{DDJ} by completing the computation of all the sequential topological complexities of all symmetric products of orientable surfaces. Furthermore, it verifies the \emph{rationality conjecture}~\cite{FO} on the \emph{$\TC$-generating function} of the symmetric products of orientable surfaces. 

The rationality conjecture in its original form can be stated as follows.

\begin{conjec}[\protect{Farber--Oprea}]\label{foconjec}
    For a finite CW complex $X$, the formal power series 
   \[
   f_X(t)=\sum_{m=1}^{\infty} \TC_{m+1}(X)\hspace{1mm}t^m
   \]
   of the $\TC$-generating function of $X$ represents a rational function of the form 
   \[
   \frac{P_X(t)}{(1-t)^2},
   \]
   where $P_X(t)$ is an integral polynomial with $P_X(1)=\cat(X)$. 
\end{conjec}

This conjecture is known for only a few classes of finite CW complexes and closed manifolds (see Section~\ref{tcm}), and its second part, which states that $P_X(1)=\cat(X)$, has recently been disproven in~\cite{FKS} using a finite $11$-dimensional CW complex. However, the original conjecture is still open for closed manifolds. 

Using Theorem~\ref{aa}, we show that Conjecture~\ref{foconjec} is verified for the symmetric products of orientable surfaces (and for their finite products --- see Section~\ref{products section}). This extends the class of closed manifolds for which the rationality conjecture holds. 

\begin{swcor}\label{aa1}
For any integers $n\ge 2$ and $g\ge 0$, the formal power series of the $\TC$-generating function of the closed orientable $2n$-manifold $SP^n(M_g)$ represents the rational function
\[
   \frac{(2t-t^2)\cat(SP^n(M_g))}{(1-t)^2}.
\]
\end{swcor}

The notions of \emph{LS-logarithmic} and \emph{$\TC_m$-logarithmic} families of spaces, on whose finite sub-families the respective product inequalities for $\cat$ and $\TC_m$ are equalities, were introduced in~\cite{AGO} to (in part) aid in the study of $\TC$-generating functions (see Definition~\ref{tclogdefn}). It turns out that the families of symmetric products of closed orientable surfaces are both LS-logarithmic and $\TC_m$-logarithmic for each $m\ge 2$.

\begin{swcor}\label{aa2}
        The family $\{SP^n(M_g)\}_{n\ge 1,g\ge 0}$ is LS-logarithmic. In particular, we have for all $m,n\ge 1$ and $g\ge 0$ that
    \[
    \cat((SP^n(M_g))^m)=m\cat(SP^n(M_g)).
    \]
\end{swcor}

In~\cite{LS}, $\TC_m(X)=\cat(X^{m-1})$ was shown for $X$ an $H$-space, and examples of spaces $X$ for which $\TC_m(X)=\cat(X^m)$ holds for each $m\ge 2$ were asked. Due to our results, the manifolds $SP^n(M_g)$ serve as such examples for all $n$ and $g$, except for the case $n=g=1$ when the manifold is the $2$-torus, which is an $H$-space.

Another consequence of (the proof of) Theorem~\ref{aa} is the following.

\begin{swcor}\label{aa3}
        The family $\{SP^n(M_g)\}_{n\ge 2,g\ge 0}$ is $\TC_m$-logarithmic for all $m\ge 2$. In particular, we have for all $m,n\ge 2$, $k\ge 1$, and $g\ge 0$ that
    \[
    \TC_m((SP^n(M_g))^k)=k\TC_m(SP^n(M_g)).
    \]
\end{swcor}

As noted in~\cite[Page 3]{AGO}, families of spaces that are both LS-logarithmic and $\TC_m$-logarithmic for all $m\ge 2$, and whose members $X$ satisfy $\TC_m(X)=m\cat(X)$, are rare. Our results show that symmetric products of orientable surfaces form such a family, whose any finite sub-family satisfies the hypotheses of~\cite[Theorem~1.5]{AGO}, thereby giving computations for $\TC_m$ of some \emph{general polyhedral product spaces}.

After studying the symmetric products of orientable surfaces, we turn towards the symmetric products of closed non-orientable surfaces $N_g$ of genus $g\ge 1$, which, compared to their orientable counterparts, have been studied significantly less in the literature. Our next main result is the exact computation of the LS-category of the symmetric products of non-orientable surfaces.

\begin{swtheorem}\label{bb}
    For any integers $n,g\ge 1$, we have for the Lusternik--Schnirelmann category of the $n$-th symmetric product of $N_g$ that
    \[
    \cat(SP^n(N_g))=2n.
    \]
\end{swtheorem}

Theorem~\ref{bb} is proved by studying the $\Z_2$-cohomology ring of $SP^n(N_g)$, which was described by Kallel and Salvatore in~\cite{KS}. Basically, we show that the $\Z_2$-cup-length of $SP^n(N_g)$ coincides with the dimension of $SP^n(N_g)$. This leads to the following interesting result.

\begin{swcor}\label{bb2}
The family $\{SP^n(N_g)\}_{n\ge 1,g\ge 1}$ is LS-logarithmic. In particular, we have for all $m,n,g\ge 1$ that
    \[
    \cat((SP^n(N_g))^m)=m\cat(SP^n(N_g))=2mn.
    \]
\end{swcor}

The LS-category of the symmetric products of orientable surfaces, as mentioned in Theorem~\ref{aa}, is determined using some results of~\cite{DDJ} and~\cite{Dr3} (see Section~\ref{cohomg}). Hence, Theorem~\ref{bb} completes the computation of the LS-category of all symmetric products of all closed surfaces, thereby giving a complete list of closed surfaces $S$ for which $SP^n(S)$ is \emph{essential}, in the sense of Gromov~\cite{Gr} (see Section~\ref{lstcsect}).

\begin{swcor}\label{bb1}
    Given a closed surface $S$ and an integer $n\ge 1$, the $n$-th symmetric product $SP^n(S)$ of $S$ is an essential $2n$-manifold if and only if either $S$ is an orientable surface of genus $\ge n$ or $S$ is a non-orientable surface of genus $\ge 1$.
\end{swcor}

In this paper, more generally, we study lower bounds to $\TC_m(SP^n(X))$ and $\cat(SP^n(X))$ for finite CW complexes $X$. 

While the LS-category of products of a space $Y$ typically gives sharp bounds to $\TC_m(Y)$, the most useful lower bounds to $\TC_m(Y)$ come from the cohomology of $Y$ and $Y^m$. Indeed, for any ring $R$, the $m$-th $R$-zero-divisor cup-length of $Y$, denoted $\zcl^m_R(Y)$, is a lower bound to $\TC_m(Y)$. We recall that $\zcl_R^m(Y)$ is the cup-length of the kernel of the homomorphism induced in $R$-cohomology by the diagonal map $\Delta_m:Y\to Y^m$. If $Y$ is a symmetric product $SP^n(X)$ of a space $X$, which is the regime of our interest, then its cohomology ring can be complicated. In particular, computing the $m$-th $R$-zero-divisor cup-length of $SP^n(X)$ can be quite difficult. Therefore, we want to provide a useful lower bound to $\TC_m(SP^n(X))$ in terms of the cohomology of $X$ and $X^m$. 

For this purpose, a natural candidate could be $\zcl_{R}^m(X)$. However, we do not know whether $\zcl^m_{R}(X)$ is a lower bound to $\TC_m(SP^n(X))$ in general. So, instead of considering all the $m$-th $R$-zero-divisors of $X$, we consider only those ``special'' zero-divisors that are obtained from cohomology classes of $X$ via some homomorphisms induced in cohomology using the projection maps $\pi_i:X^m\to X$, see Definition~\ref{szcldefn}. We denote by $\szcl^m_R(X)$ the length of the longest non-trivial cup product of such \emph{special zero-divisors} of $X$. Of course, $\szcl^m_R(X)\le \zcl^m_R(X)$.

It turns out that for any finite CW complex $X$, the rational cup-length and the $m$-th rational special zero-divisor cup-length of $X$ bound $\cat(SP^n(X))$ and $\TC_m(SP^n(X))$, respectively, from below. In fact, they offer sharp lower bounds to some other invariants of interest that bound $\cat(SP^n(X))$ and $\TC_m(SP^n(X))$ from below. This observation is encapsulated in the following statement.

\begin{swtheorem}\label{cc}
    For any finite CW complex $X$, ring $R$, and integer $n\ge 1$, we have that $\cu_R(X)\le \cu_R(SP^n(X))$ and $\cu_{\Q}(X)= \cu_{\Q}(\textup{Im}(\delta_n^*))$. Moreover, for any integer $m\ge 2$, we have that $\szcl^m_{\Q}(\textup{Im}((\delta_n^m)^*))=\szcl^m_{\Q}(X)\le \szcl_{\Q}^m(SP^n(X))$.
\end{swtheorem}

Here, $\delta_n:X\to SP^n(X)$ is the diagonal embedding defined as $\delta_n(x)=[x,\ldots,x]$, $\delta_n^m:X^m\to (SP^n(X))^m$ is the product map defined as 
\[
\delta_n^m(x_1,\ldots,x_m)=\left(\delta_n(x_1),\ldots,\delta_n(x_m)\right),
\]
and $\szcl^m_{\Q}(\textup{Im}(\left(\delta_n^m\right)^*))$ is the length of the longest non-trivial cup product of those $m$-th rational special zero-divisors of $X$ that belong to the subring $\textup{Im}(\left(\delta_n^m\right)^*)$.

There is another lower bound to the $m$-th sequential topological complexity of a space, namely its $m$-th sequential distributional complexity $(\dTC_m)$, which is a new homotopy invariant of probabilistic flavor that also has some very interesting interpretations in terms of motion planning, see~\cite{DJ},~\cite{KW},~\cite{J1}. 

While studying lower bounds to sequential distributional complexities of CW complexes in~\cite{DJ}, symmetric products of these spaces and their cohomology rings came into the picture. More precisely, some general cohomological lower bounds to $\dTC_m(X)$ were found in~\cite{J1} (see~\cite{DJ} for $m=2$) in terms of the cohomology of $SP^n(X)$ and $SP^n(X^m)$ using the diagonal embedding $X^m\to SP^n(X^m)$. So, it becomes very tempting to explore relationships between lower bounds to the $m$-th sequential distributional complexity of spaces and the $m$-th sequential topological complexity of their diagonal embeddings into symmetric products for each $m\ge 2$. In that direction, we show that these two homotopy invariants share a common cohomological lower bound, namely the $m$-th rational special zero-divisor cup-length of the space. We do this by finding sharp lower bounds to sequential distributional complexities of continuous maps. The precise statement of our general result in Alexander--Spanier cohomology is as follows.

\begin{swtheorem}\label{dd}
Let $f:X\to Y$ be a continuous map, $m\ge 2$, and $R$ be a ring with unity. For $1 \le i \le n$ and some integers $k_i\ge 1$, let $\alpha_i^*\in H^{k_i}((SP^{n!}(Y))^m;R)$ be $m$-th $R$-zero-divisors of $SP^{n!}(Y)$. Let $\alpha_i\in H^{k_i}(X^m;R)$ be their images under $(\delta_{n!}^m\circ f^m)^*:H^*((SP^{n!}(Y))^m;R)\to H^*(X^m;R)$. If $\alpha_1\smile\cdots\smile\alpha_n\ne 0$, then we have for the $m$-th sequential distributional complexity of the map $f$ that
\[
\dTC_m(f)\ge n.
\]
\end{swtheorem}

While a cohomological lower bound to $\dTC_m(f)$ can be readily obtained in terms of the cohomology of $SP^{k}(Y^m)$ by generalizing the proofs of~\cite[Section 4]{J1} from spaces to maps, the relevance of the lower bound obtained in Theorem~\ref{dd} is that in general, the cohomology of $(SP^{k}(Y))^m$ is simpler\hspace{0.2mm}\footnote{\hspace{0.5mm}To see this, take $Y$ to be a circle or a closed surface and use Künneth formulae --- see~\cite{Mac} and Section~\ref{cohomg} in the case of orientable surfaces, and~\cite{KS} and Section~\ref{cohong} otherwise.} than that of $SP^{k}(Y^m)$ when $m,k\ge 2$. Hence, Theorem~\ref{dd} makes it easier to bound $\dTC_m(f)$ from below. Furthermore, it helps obtain the following convenient lower bound to $\dTC_m(X)$.

\begin{swcor}\label{dd1}
For any finite CW complex $X$ and integer $m\ge 2$, we have for the $m$-th sequential distributional complexity of $X$ that $\szcl^m_{\Q}(X)\le \dTC_m(X)$.
\end{swcor}

Corollary~\ref{dd1} subsumes the computations of $\dTC_m$ for orientable surfaces and products of spheres done in~\cite[Section 7]{J1}, which are rather long and complicated. It also gives the following for the symmetric products of closed orientable surfaces.

\begin{swcor}\label{dd2}
    For any integers $m\ge 2$, $n\ge 1$, and $g\ge 0$, we have the equalities $\dTC_m(SP^n(M_g))=\TC_m(SP^n(M_g))$ and $\dcat((SP^n(M_g))^m)=\cat((SP^n(M_g))^m).$
\end{swcor}

Some other results in this paper on the homotopy of the symmetric products of closed surfaces, which we now discuss, could be of interest on their own. 

The Dold--Thom theorem~\cite{DT} gives an isomorphism $\pi_i(SP^{\infty}(X))=\wt{H}_i(X;\Z)$ for each $i$, where $SP^{\infty}(X)$ is the infinite symmetric product of $X$. It is natural to ask if there is a choice of integer $n\ge 2$ such that the Dold--Thom isomorphism can be realized at the level of $SP^n(X)$. Indeed, for simply connected $X$, it was shown in~\cite{DP} that $\pi_i(SP^n(X))= \wt{H}_i(X;\Z)$ for $n>\tfrac{i-\text{conn}(X)}{2}$, where $\text{conn}(X)$ denotes the homotopy connectivity of $X$. In this paper, we address this question for when $X$ is a closed surface $S$ (that is not simply connected) by proving the following.

\begin{swtheorem}\label{ee}
Given a closed surface $S$ and integer $i\ge 1$, there is an isomorphism
\[
\pi_i(SP^n(S))=\wt{H}_i(S;\Z)
\]
for all $n>i$. In particular, the $i$-th homotopy group of the $n$-th symmetric product of $S$ vanishes for each $3\le  i  <  n$, $\pi_2(SP^n(M_g))=\Z$, and $\pi_2(SP^n(N_g))=0$.
\end{swtheorem}

The case $i=1$ in the above theorem is well-known, see, for example,~\cite{KT}. The case $i=2$ for $S=M_g$ was studied in~\cite{BR}, and more recently in~\cite{DDJ} using a cellular complex construction from~\cite{KS}. Here, we obtain Theorem~\ref{ee} using the Dold--Thom homotopy splitting~\cite{DT} of the infinite symmetric product of some finite $2$-dimensional CW complexes and the work of~\cite{KS}.  

In particular, Theorem~\ref{ee} implies that, while $M_{g-1}$ is the orientable double cover of $N_g$ for each $g\ge 1$, a finite cover of $SP^n(M_h)$ cannot be a finite cover of $SP^n(N_g)$ for any choice of integers $n\ge 3$, $g\ge 1$, and $h\ge 0$. 

In~\cite{DDJ}, it was shown that the universal cover of the orientable manifold $SP^n(M_g)$ admits a spin structure if and only if $n-g$ is odd. In contrast, as a result of Theorem~\ref{ee}, we conclude the following in the non-orientable setting.

\begin{swcor}\label{ee1}
For any integers $n\ge 3$ and $g\ge 1$, the universal cover $\wt{SP^n(N_g)}$ of the non-orientable manifold $SP^n(N_g)$ admits a spin structure.
\end{swcor}

This paper is organized as follows. In Section~\ref{Preliminaries}, we mention some facts about symmetric products, sequential topological complexity, and sequential distributional complexity. We study the cohomology of symmetric products of finite CW complexes in Section~\ref{secszcl}, where we define special zero-divisors and prove Theorem~\ref{cc}. In Section~\ref{appl}, we prove Theorem~\ref{dd} and Corollary~\ref{dd1}. We give a proof of our main result (Theorem~\ref{aa}) in Section~\ref{symorient}. In the subsections of this section, we also prove Corollaries~\ref{aa1},~\ref{aa2},~\ref{aa3}, and~\ref{dd2}, and mention several other consequences of our results. We begin Section~\ref{symnonorient} by proving Theorem~\ref{ee} and Corollary~\ref{ee1}. Then we give proofs of our second main result (Theorem~\ref{bb}) and of Corollaries~\ref{bb2} and~\ref{bb1}. We end this paper by discussing the problem of determining the (sequential) topological complexity of symmetric products of closed non-orientable surfaces.

\subsection*{Notations and conventions} In this paper, the term \emph{space} refers to a connected CW complex, \emph{map} refers to a topologically continuous function or a group homomorphism, \emph{surface} refers to a closed connected $2$-manifold, and \emph{ring} refers to a commutative ring with unity. We use the symbol ``$=$'' to denote homeomorphisms of spaces and isomorphisms of groups, and ``$\simeq$'' to denote homotopy equivalences. The tensor product is taken over $\Z$ and is denoted by $\otimes$.

\section{Preliminaries}\label{Preliminaries}
In this section, we give a brief review of the symmetric products of spaces and the classical and distributional complexities of spaces and maps.

\subsection{Symmetric products}\label{symprod}
The $n$-th symmetric product of a space $X$, denoted $SP^n(X)$, is defined as the orbit space $X^n/S_n$ of the natural action of the symmetric group $S_n$ on the product space $X^n$ which permutes the coordinates. Then the elements of $SP^n(X)$ are unordered $n$-tuples $[x_1,\dots, x_n]$, where $x_i\in X$. We write $\vartheta_n:X^n\to SP^n(X)$ for the corresponding quotient map. 

Let us fix some $x_0\in X$ as the basepoint of $X$. Then there are the basepoint inclusions $SP^n(X)\hookrightarrow SP^{n+k}(X)$ for each $k\ge 1$. The colimit over the symmetric products with respect to the basepoint inclusions defines the infinite symmetric product $SP^{\infty}(X)$, which is a free abelian topological monoid. It follows from the Dold--Thom theorem~\cite{DT} that $\pi_k(SP^{\infty}(X))=\wt{H}_k(X;\Z)$ for each $k\ge 1$.

When $X$ is a CW complex, $SP^n(X)$ can also be given a CW structure. Indeed, the product space $X^n$ has a CW structure such that each $\sigma\in S_n$ permutes the cells.
This CW structure on $X^n$ respects the permutation action of the symmetric group $S_n$, thereby giving a CW complex structure to the orbit space $SP^n(X)$, see \cite[Proposition II.1.5 and Exercise II.1.17 (2)]{tD}. 

If $X$ is a $2$-dimensional CW complex with only one vertex (for example, if $X$ is a surface), then the CW structure on $SP^n(X)$ is natural and can be described explicitly~\cite{KS}, see also~\cite{BR}. More precisely, for some integers $k\ge 0$ and $r\ge 1$, let $X$ be a $2$-dimensional CW complex obtained by attaching $r$ number of closed $2$-discs to the wedge $\bigvee_{i=0}^k\hspace{0.5mm}S^1_i$. Let $e\in X$ be the basepoint of the wedge that is also the identity element of each circle group $S^1_i$. Define a relation $\sim$ on $SP^n(X)$ such that
\[
\left[x_1,x_2,x_3,\ldots,x_n\right]\sim \left[e,x_1\cdot x_2,x_3,\ldots,x_n\right]
\]
whenever $x_1,x_2\in S^1_i$ for some $i$. Here, $\cdot$ denotes the group-theoretic multiplication on $S_i^1$. Let $\ov{SP}^n(X):=SP^n(X)/\sim$ and $q_n:SP^n(X)\to \ov{SP}^n(X)$ be the quotient map. It follows from~\cite{KS} that $\ov{SP}^n(X)$ has a canonical CW structure and $q_n$ is a natural homotopy equivalence. Also, the composition $q_n \circ\vartheta_n:X^n\to \ov{SP}^n(X)$ is cellular. The basepoint inclusions $\ov{SP}^n(X)\hookrightarrow \ov{SP}^{n+1}(X)$ define $\ov{SP}^{\infty}(X)$ in the obvious way. Since $q_n\circ\vartheta_n$ is cellular for each $n$, the $k$-skeleton of $\ov{SP}^k(X)$ coincides with the $k$-skeleton of $\ov{SP}^n(X)$ for each $k\le n$. In particular, this means that
\begin{equation}\label{skeleton}
    \text{the } k\text{-skeleton of } \ov{SP}^{\infty}(X) \text{ coincides with that of } \ov{SP}^n(X) \text{ for each } n\ge k.
\end{equation}

If $X$ is a closed manifold (i.e., a compact manifold without boundary), then it is not difficult to show that $SP^n(X)$ is a closed manifold for each $n\ge 2$ if and only if $\dim(X)=2$, see~\cite{Pu} for a proof sketch. For this reason, we will focus more on surfaces and $2$-dimensional CW complexes of the above kind in Sections~\ref{symorient} and~\ref{symnonorient} of this paper for explicit results and computations.

\subsection{LS-category and topological complexity}\label{lstcdefn}
Let $X$ be a space and $R$ be a ring. 
Then the cup-length of $H^*(X;R)$ (or the $R$-cup-length of $X$), denoted $\cu_R(X)$, is the maximal length $k$ of a non-zero cup product $\alpha_1\smile\cdots\smile\alpha_k\ne 0$ of positive degree cohomology classes $\alpha_i\in H^*(X;R)$.

We recall that the \emph{Lusternik--Schnirelmann category} (LS-category) of a map $f:X\to Y$, denoted $\cat(f)$, is the smallest integer $n$ for which there is a covering $\{U_i\}$ of $X$ by $n+1$ open sets such that the restricted maps $f_{|U_i}:U_i\to Y$ are null-homotopic for each $i$,~\cite{BG}. When $f$ is the identity map on $X$, we obtain the LS-category of the space $X$, denoted $\cat(X)$,~\cite{Ja},~\cite{CLOT}.

For a proof of the following, see~\cite{BG} and~\cite{Ja} (see also~\cite{Sch} and~\cite{CLOT} for the case when $f$ is the identity map).

\begin{thm}\label{clbound}
    For a map $f:X\to Y$ between finite-dimensional CW complexes and ring $R$, $\cu_R(\textup{Im}(f^*))\le\cat(f)\le\min\{\cat(X),\cat(Y)\}\le\min\{\dim(X),\dim(Y)\}$.
\end{thm}

For any integer $m\ge 2$, let $\Delta_m:X\to X^m$ denote the diagonal map. Suppose $\Delta_m^*:H^*(X^m;R)\to H^*(X;R)$ is the map induced by $\Delta_m$ in $R$-cohomology. 

\begin{defn}[\protect{\cite{BGRT}}]
    The $m$-th \emph{$R$-zero-divisor cup-length} of $X$, denoted $\zcl^m_R(X)$, is defined as the cup-length of the ideal $\text{Ker}(\Delta_m^*)$.
\end{defn}

For given $m \ge 2$, the $m$-th \emph{sequential topological complexity} of a map $f:X\to Y$, denoted $\TC_m(f)$, is the smallest integer $n$ for which there exists a covering $\{U_i\}$ of $X^m$ by $n+1$ open sets over each of which there is a map $s_i : U_i \to P(Y)$ such that for each $(x_1,\ldots,x_m) \in U_i \subset X^m$, 
\[
 s_i\left(x_1,\ldots,x_m\right)\left(\frac{j-1}{m-1}\right) = f\left(x_j\right) 
\]
for all $1 \le j \le m$, see~\cite{Ku} (and also~\cite{Sco} for the case $m=2$). Here, $P(Y)=Y^{[0,1]}$ denotes the path space of $Y$ with the compact-open topology. When $f$ is the identity map on $X$, we obtain the $m$-th sequential topological complexity of the space $X$, denoted $\TC_m(X)$, which was introduced by Rudyak~\cite{Ru} as an extension to Farber's~\cite{Far1} original notion of topological complexity $\TC(X)$.

The next theorem follows from~\cite{Sch} and~\cite{Ru} (and also~\cite{Far1} in the case $m=2$).

\begin{thm}\label{tcbound}
    For a space $X$, integer $m\ge 2$, and ring $R$, we have the inequalities $\max\{\cat(X^{m-1}),\zcl^m_R(X)\}\le\TC_m(X)\le\cat(X^m)\le m\cat(X)$.
\end{thm}

The proof of the following statement, which generalizes Theorem~\ref{tcbound}, can be found in~\cite{Ku} (see also~\cite{Sco} for the case $m=2$).

\begin{thm}\label{tcbound2}
    For a map $f:X\to Y$, integer $m\ge 2$, and ring $R$, we have that $\max\{\cat(f^{m-1}),\zcl^m_R(\textup{Im}((f^m)^*))\}\le \TC_m(f)\le\min\{\cat(f^m),\TC_m(X),\TC_m(Y)\}$.
\end{thm}

Here, $f^k:X^k\to Y^k$ is defined as $f^k(x_1,\ldots,x_k)=(f(x_1),\ldots,f(x_k))$ for each $k\ge 1$, and $\zcl_R^m(\textup{Im}((f^m)^*)):=\cu_R(\textup{Im}((f^m)^*)\cap \text{Ker}(\Delta_m^*))$.

\subsection{Distributional category and complexity}\label{2c}
Recently in~\cite{DJ} and~\cite{J1}, some probabilistic variants of $\cat$ and $\TC_m$ were introduced. These are the distributional category and the $m$-th sequential distributional complexity, respectively. Some other closely related probabilistic variants, called the ``analog invariants'', were also introduced and studied, independently, by Knudsen and Weinberger in \cite{KW} around the same time. See~\cite[Section 1]{DJ},~\cite[Section 1]{KW}, and~\cite[Section 2.A]{J2} for some motivations behind these new invariants.

For a metric space $Z$, let $\B(Z)$ denote the set of probability measures on $Z$. Also, for any $n \ge 1$, let $\mathcal{B}_{n}(Z) = \{\mu \in \mathcal{B}(Z) \mid |\supp(\mu)| \le n \}$ denote the space of measures on $Z$ supported by at most $n$ points, equipped with  the Lévy--Prokhorov metric,~\cite{Pr},~\cite[Section 3.1]{DJ}. Let $P(z_1,z_2)=\{\phi\in P(Z)\mid \phi(0)=z_1,\phi(1)=z_2\}$ for any $z_1,z_2\in P(Z)$. If $Z$ has a fixed basepoint $z_0\in Z$, then we write the based path space of $Z$ as $P_0(Z)=\{\phi\in P(Z)\mid \phi(1)=z_0\}$.

We recall that the \emph{distributional category} of a pointed map $f:(X,x_0)\to (Y,y_0)$, denoted $\dcat(f)$, is the smallest integer $n$ for which there exists a continuous map $H:X\to \B_{n+1}(P_0(Y))$ such that $H(x)(0)\in \B_{n+1}(P(f(x),y_0))$ for each $x\in X$, see~\cite{Dr4},~\cite{J2}. When $f$ is the identity map on $X$, we obtain the distributional category of the space $X$, denoted $\dcat(X)$,~\cite{DJ}. 

In~\cite{DJ}, it was shown that $\cu_{\Q}(X)\le \dcat(X)\le\cat(X)$ when $X$ is a finite CW complex. This inequality can be generalized from spaces to maps as follows. 

\begin{thm}\label{dcatbound}
    For a map $f:X\to Y$ between finite CW complexes, we have that $\cu_{\Q}(\textup{Im}(f^*))\le\dcat(f)\le\min\{\cat(f),\dcat(X),\dcat(Y)\}$.
\end{thm}

\begin{proof}[Proof sketch]
    Only the lower bound $\cu_{\Q}(\textup{Im}(f^*))\le\dcat(f)$ needs some explanation. It follows from~\cite[Proposition 2.5 (4)]{J2} that $\cu_{\Q}(\textup{Im}(f^*\circ \delta_{n!}^*))\le\dcat(f)$, where $\delta_{n!}:Y\to SP^{n!}(Y)$ is the diagonal embedding (see the Introduction). By Proposition~\ref{split2}, $\delta_{n!}^*$ is an epimorphism, so one can proceed along the lines of~\cite[Sections 4.1 \& 4.2]{DJ} to conclude that $\cu_{\Q}(\textup{Im}(f^*\circ \delta_{n!}^*))=\cu_{\Q}(\textup{Im}(f^*))$.
\end{proof}

Let us fix $m\ge 2$. For any $m$-tuple $(z_1,\ldots,z_m)\in Z^m$, let 
\[
P\left(z_1,\ldots,z_m\right)=\left\{ \phi\in P(Z)\hspace{1mm}\middle|\hspace{1mm} \phi\left(\frac{i-1}{m-1}\right)=z_i \textup{ for all } 1 \le i \le m\right\}.
\]

Given a map $f:X\to Y$ and integer $m \ge 2$, the $m$-th \emph{sequential distributional complexity} of $f$, denoted $\dTC_m(f)$, is the smallest integer $n$ for which there exists a map $s_m:X^m\to\B_{n+1}(P(Y))$ such that for each $(x_1,\ldots,x_m)\in X^m$, $s_m(x_1,\ldots,x_m)\in\B_{n+1}(P(f(x_1),\ldots,f(x_m)))$. When $f$ is the identity map on $X$, we obtain the $m$-th sequential distributional complexity of $X$, denoted $\dTC_m(X)$, see~\cite{J1} (and also~\cite{DJ} for the case $m=2$).

The proof of the following statement is a generalization of the proof of the corresponding statement for $\dTC_m(X)$ found in~\cite[Section 3]{J1}.

\begin{thm}\label{dtcbound}
    For a map $f:X\to Y$, integer $m\ge 2$, and ring $R$, we have that $\dcat(f^{m-1})\le \dTC_m(f)\le\min\{\dcat(f^m),\TC_m(f),\dTC_m(X),\dTC_m(Y)\}$.
\end{thm}

In general, the distributional invariants are different from their classical counterparts. For example, $\dcat(\R P^n)=\dTC(\R P^n)=1<n=\cat(\R P^n)\le\TC(\R P^n)$ for each $n\ge 2$, see~\cite[Section 3]{DJ} (and compare with~\cite[Example 1.8 (2)]{CLOT} and \cite[Theorem 12]{FTY}). For various other interesting differences, see~\cite{KW},~\cite{Dr4},~\cite{J1}.

\section{Symmetric products of finite CW complexes}\label{secszcl}
In this section, motivated by the problem of bounding from below the LS-category and the sequential topological complexities of the symmetric products of a finite CW complex $X$ in terms of the cohomology of $X$, we prove Theorem~\ref{cc}.

Let $X$ be a finite CW complex and $x_0\in X$ be a fixed basepoint. Let us write $\xi_n:X\to SP^n(X)$ for the basepoint inclusion defined as $\xi_n(x)=[x,x_0,\ldots,x_0]$, where $x_0$ repeats $(n-1)$-times. Then we have the following classical result, which is often credited to Steenrod.

\begin{theorem}[\protect{\cite{Na},~\cite{Do1},~\cite{Do2}}]\label{split1}
    For a finite CW complex $X$ and $n\in \mathbb{N}\cup\{\infty\}$, the map $\xi_n$ induces a split epimorphism $\xi_n^*:H^*(SP^n(X);R)\to H^*(X;R)$ in $R$-cohomology for any ring $R$.
\end{theorem}

Let $\delta_n:X\to SP^n(X)$ be the diagonal embedding defined as $\delta_n(x)=[x,x,\ldots,x]$. Using Theorem~\ref{split1} and the Dold--Thom theorem, the following result was obtained in~\cite[Section 4]{DJ}.

\begin{prop}\label{split2}
    For a finite CW complex $X$ and $n\in \mathbb{N}$, the map $\delta_n$ induces a split epimorphism $\delta_n^*:H^*(SP^n(X);\Q)\to H^*(X;\Q)$ in rational cohomology.
\end{prop}

Note that the proof of Proposition~\ref{split2} from~\cite[Proposition 4.3]{DJ} also works for cohomology with coefficients in $\R$. So, the results of this section that we will prove using Proposition~\ref{split2} in $\Q$-coefficients will also be true for $\R$-coefficients.

For our convenience, we now define the following.

\begin{defn}\label{szcldefn}
    Given a space $X$, ring $R$, and integer $m\ge 2$, we say that a class $u\in \Ker(\Delta_m^*:H^*(X^m;R)\to H^*(X;R))$ is \emph{special} if $u$ is the image of a direct sum
    \[
    \bigoplus_{j=1}^{p} \alpha_jx, \hspace{3mm} \text{where }\ 2\le p\le m, \text{ } \alpha_j\in \Z\setminus\{0\},  \ \text{and}\ x\in H^*(X;R),
    \]
    under a map $\phi:\oplus_{j=1}^{p} H^*(X;R)\to H^*(X^m;R)$ that is defined as a direct sum of the maps $\pi_j^*:H^*(X;R)\to H^*(X^m;R)$ induced, respectively, by the projections $\pi_j:X^m\to X$ onto the $j$-th summands, i.e.,
    \[
    \phi = \bigoplus_{j=1}^{p} \beta_j \hspace{1mm}\pi^*_{\sigma(j)}
    \]
    for some integers $\beta_j\ne 0$ and injective function $\sigma:\{1,\ldots,p\}\to\{1,\ldots,m\}$. 
    
    We call such a class $u\in H^*(X^m;R)$ an \emph{$m$-th $R$-special zero-divisor} of $X$. Then we define the \emph{$m$-th $R$-special zero-divisor cup-length} of $X$, denoted $\szcl_R^m(X)$, to be the largest integer $k$ for which there exist $k$ number of positive degree $m$-th $R$-special zero-divisors $u_i$ such that $u_1\smile \cdots \smile u_k\ne 0$. 
\end{defn}

By Theorem~\ref{tcbound}, $\szcl_R^m(X)\le \zcl_R^m(X)\le \TC_m(X)$. It is also easy to see that if $X\simeq Y$, then $\szcl_R^m(X)= \szcl_R^m(Y)$ for each ring $R$ and integer $m\ge 2$.

\begin{remark}
    For $m=2$, one usually takes $p=2$, $\phi=\pi_2^*-\pi_1^*$, and $\alpha_1=\alpha_2=1$ in the notations of Definition~\ref{szcldefn} --- see, for instance, the proofs of~\cite[Theorems 8, 9, 10, \& 13]{Far1} and~\cite[Theorem 1]{FTY}.
\end{remark}

\begin{ex}\label{szcleg}
    \begin{enumerate}
        \itemsep-0.30em 
        \item Let $X=\prod_{i=1}^nS^{k_i}$ for some $n\ge 1$ and $k_i\ge 1$. Then, $\TC_m(X)=\szcl_{\Q}^m(X)=n(m-1)+l_n$, where $l_n$ is the number of $k_i$ that are even, see~\cite[Theorem 3.10 \& Corollary 3.12]{BGRT}. 

        \item For any $g\ge 2$, $\TC_m(M_g)=\szcl_{\Q}^m(M_g)=2m$, see~\cite[Prop 3.2]{GGGHMR}.

        \item More generally, for $X=M_g\times \prod_{i=1}^nS^{k_i}$ where $g\ge 2$ and $n,k_i\ge 1$, $\TC_m(X)=\szcl_{\Q}^m(X)=m(n+2)-n+l_n$.
    \end{enumerate}
For an explicit description of the special zero-divisors involved in the above examples, we refer the reader to the proofs in~\cite[Section 7]{J1}, where the notion of the rational special zero-divisors was implicitly used. 
\end{ex}

We now obtain lower bounds to $\cat(SP^n(X))$ and $\TC_m(SP^n(X))$ in terms of the cohomology of $X$ and $X^m$, respectively.

\begin{proof}[Proof of Theorem~\ref{cc}]
    Let $\cu_R(X)=r$. Then there exist $\alpha_i\in H^{k_i}(X;R)$ for some $k_i>0$ such that $\alpha_1\smile\cdots\smile\alpha_r\ne 0$. By Theorem~\ref{split1}, there exist cohomology classes $\beta_i\in H^{k_i}(SP^n(X);R)$ such that $\xi_n^*(\beta_i)=\alpha_i$ for each $i$. Because,  the map $\xi_n^*$ is multiplicative, we must have
    \[
    \beta_1\smile\cdots\smile\beta_r\ne 0.
    \]
    Hence, $\cu_R(X)=r\le \cu_R(SP^n(X))\le\cat(SP^n(X))$ in view of Theorem~\ref{clbound}. In the case $R=\Q$, Proposition~\ref{split2} gives a section $s:H^*(X)\to H^*(SP^n(X))$ to the map $\delta_n^*$. So, $\alpha_i$ has a pre-image under $\delta_n^*$, namely $s(\alpha_i)\in H^{k_i}(SP^n(X))$. Therefore, we have $\cu_{\Q}(X)=r\le \cu_{\Q}(\text{Im}(\delta_n^*))$. The inequality $\cu_{\Q}(\text{Im}(\delta_n^*))\le \cu_{\Q}(X)$ is straightforward because $\text{Im}(\delta_n^*)$ is a subring of $H^*(X;\Q)$.

Next, let $\szcl^m_{\Q}(X)=k$. We have special zero-divisors $u_i\in H^*(X^m;\Q)$ such that $u_1\smile\cdots\smile u_k\ne 0$. For each $1\le i\le k$, suppose
    \[
    u_i=\phi_i\left(\bigoplus_{j=1}^{p_i} \alpha_j^i\hspace{1mm}x^i\right) \hspace{2.5mm} \text{for some}\hspace{2mm} \phi_i = \bigoplus_{j=1}^{p_i} \beta_j^i \hspace{1mm}\pi^*_{\sigma_i(j)},
    \]
in notations of Definition~\ref{szcldefn}, where $2\le p_i\le m$ and $\sigma_i:\{1,\ldots,p_i\}\to\{1,\ldots,m\}$ is an injective function. By Proposition~\ref{split2}, there exists $y^i\in H^*(SP^n(X);\Q)$ such that $\delta_n^*(y^i)=x^i\in H^*(X;\Q)$ for each $i$. Let $r_t:(SP^n(X))^m\to SP^n(X)$ be the projection onto the $t^{th}$ factor for each $1 \le t \le m$. Corresponding to each $u_i$ and $\phi_i$, define the map $\psi_i:\oplus_{j=1}^{p_i}H^*(SP^n(X);\Q)\to H^*((SP^n(X))^m;\Q)$ such that
    \[
    \psi_i = \bigoplus_{j=1}^{p_i} \beta_j^i \hspace{1mm} r^*_{\sigma_i(j)}, \hspace{2.5mm}\text{and let} \hspace{2mm} v_i=\psi_i\left(\bigoplus_{j=1}^{p_i} \alpha_j^i\hspace{1mm}y^i\right).
    \]
    For each $i$, the following diagram commutes, where $\pa_m:SP^n(X)\to (SP^n(X))^m$ denotes the diagonal inclusion and $\delta_n^m:X^m\to(SP^n(X))^m$ is the product map defined as $\delta_n^m(x_1,\ldots,x_m)=(\delta_n(x_1),\ldots,\delta_n(x_m))$.
    \begin{equation}\label{one11}
\begin{tikzcd}[contains/.style = {draw=none,"\in" description,sloped}]
\bigoplus_{j=1}^{p_i} H^*(SP^n(X);\Q)  \arrow{r}{\psi_i}  \arrow[swap]{d}{\bigoplus_{j=1}^{p_i} \hspace{0.5mm} \delta_n^* \hspace{1mm}}
&
H^*((SP^n(X))^m;\Q)  \arrow{r}{\pa_m^*} \arrow[swap]{d}{(\delta_n^m)^*}
&
H^*(SP^n(X);\Q) \arrow[swap]{d}{\delta_n^*}
\\
\bigoplus_{j=1}^{p_i} H^*(X;\Q)  \arrow{r}{\phi_i} 
&
H^*(X^m;\Q)  \arrow{r}{\Delta_m^*} 
&
H^*(X;\Q)
\end{tikzcd}
    \end{equation}
Clearly, $(\delta^m_n)^*(v_i)=u_i$ for each $i$. Hence, $v_1\smile\cdots\smile v_k\ne 0$. Also, $\pa_m^*(v_i)=0$ because $\Delta_m^*(u_i)=0$ and the top (resp. bottom) row in~\eqref{one11} when restricted on the left to any of the summand $H^*(SP^n(X);\Q)$ (resp. $H^*(X;\Q))$ is an isomorphism (we refer to the proof of~\cite[Proposition 7.1]{J1} for more details). Thus, each $v_i$ is \emph{special} and therefore, $\szcl_{\Q}^m(X)=k\le \min\{\szcl_{\Q}^m(SP^n(X)),\szcl^m_{\Q}(\text{Im}((\delta_n^m)^*))\}$. The inequality $\szcl^m_{\Q}(\text{Im}((\delta_n^m)^*))\le \szcl_{\Q}^m(X)$ is obvious.
\end{proof}

\begin{remark}
    We use special zero-divisors (in the sense of Definition~\ref{szcldefn}), instead of all zero-divisors, in the proof of Theorem~\ref{cc} primarily because of the following reason: while it can be shown using Proposition~\ref{split2} and the Künneth formulae that $(\delta_{n!}^m)^*$ is an epimorphism in rational cohomology, we don't know if the pre-images of (ordinary) zero-divisors of $X$ under $(\delta_{n!}^m)^*$ are zero-divisors of $SP^n(X)$. However, as demonstrated in the proof of Theorem~\ref{cc}, it is easy to show that the pre-images of special zero-divisors of $X$ under $(\delta_{n!}^m)^*$ are special zero-divisors of $SP^n(X)$.

    A referee informed us that zero-divisors similar to our special zero-divisors have been studied by~\cite{Ra},~\cite{CS}, who defined, for $z\in H^*(Z;\mathbb{\Q})$, ``basic zero-divisors'' $\pi_i^*(z)-\pi_j^*(z)\in H^*(Z^m;\mathbb{\Q})$ and showed that $\Ker(\Delta_m^*)$ is the ideal generated by all basic zero-divisors of $Z$. In view of this, it seems plausible that $\szcl^m_{\Q}(X)=\zcl_{\Q}^m(X)$. However, we don't pursue the inequality $\zcl_{\Q}^m(X)\le\szcl_{\Q}^m(X)$ in this paper.
\end{remark}

Theorems~\ref{cc},~\ref{clbound}, and~\ref{tcbound2} imply the following.

\begin{cor}\label{veryobv}
    For any finite CW complex $X$ and integers $m\ge 2$ and $n\ge 1$, $\cu_{\Q}(X)\le\cat(\delta_n)$ and $\szcl_{\Q}^m(X)\le\TC_m(\delta_n)$.
\end{cor}

We conclude this section by making a few simple observations.

\begin{remark}\label{sharp}
    For each $n\ge 1$, it is known that $SP^n(S^1)\simeq S^1$ (see~\cite{KS}). Hence, $\TC_m(SP^n(S^1))=\szcl_{\Q}^m(S^1)=m-1$ (see~\cite{Ru}) and $\cat(SP^n(S^1))=\cu_{\Q}(S^1)=1$. Thus, the lower bounds obtained in Theorem~\ref{cc} and Corollary~\ref{veryobv} are sharp.
\end{remark}

\begin{ex}
    Due to Corollary~\ref{veryobv}, $\TC_m(\delta_n)$ (resp. $\cat(\delta_n)$) must agree with $\TC_m(X)$ (resp. $\cat(X)$) whenever $\szcl^m_{\Q}(X)=\TC_m(X)$ (resp. $\cu_{\Q}(X)=\cat(X)$). In particular, this holds when $X$ is a simply connected symplectic manifold, a finite product of spheres, or a symmetric product of orientable surfaces. For the first two classes, see~\cite[Section 3]{BGRT}; for the last class, see proofs of Theorems~\ref{aa} and~\ref{oldknown}.
\end{ex}

\section{Sequential distributional complexity}\label{appl}

A lower bound to the $m$-th sequential distributional complexity, $\dTC_m(X)$, for each $m\ge 2$ was obtained in~\cite[Theorem 4.4]{J1} in terms of the Alexander--Spanier cohomology of $SP^{k}(X^m)$. In this section, we will prove Theorem~\ref{dd}, which gives a lower bound to $\dTC_m(X)$ in terms of the $m$-th zero-divisors of $SP^{k}(X)$.

Let us fix some integers $m,n\ge 2$. For each $1 \le i \le m$, given the projection maps $\pi_i:Y^m\to Y$, let $r_i=SP^{n!}(\pi_i):SP^{n!}(Y^m)\to SP^{n!}(Y)$ be the induced maps. Define $r:SP^{n!}(Y^m)\to(SP^{n!}(Y))^m$ as the map $r=(r_1,\ldots,r_m)$. Suppose that $\delta_{n!}:Y\to SP^{n!}(Y)$ and $\pa_m:Y^m\to SP^{n!}(Y^m)$ are the respective diagonal inclusions. Then, $r \circ\pa_m=\delta_{n!}^m$ for the product map $\delta_{n!}^m:Y^m\to (SP^{n!}(Y))^m$.

Let us fix a map $f:X\to Y$ and consider the following two pullback diagrams. 
\begin{equation}\label{two22}
\begin{tikzcd}[contains/.style = {draw=none,"\in" description,sloped}]
\mathcal{D}_{n,m} \arrow{r}{a} \arrow[swap]{d}{\sigma_n} 
& 
SP^{n!}(P(Y)) \arrow[swap]{d}{\zeta_n}
\\
X^m  \arrow{r}{\pa_m f^m}
& 
SP^{n!}(Y^m),
\end{tikzcd}
\qquad
\begin{tikzcd}[contains/.style = {draw=none,"\in" description,sloped}]
\mathcal{Z}_{n,m} \arrow{r}{b} \arrow[swap]{d}{\theta_n} 
& 
SP^{n!}(P(Y)) \arrow[swap]{d}{r \zeta_n}
\\
X^m  \arrow{r}{\delta_{n!}^m f^m}
& 
(SP^{n!}(Y))^m.
\end{tikzcd}
\end{equation}

Here, $\zeta_n=SP^{n!}(\tau_m)$, where $\tau_m:P(Y)\to Y^m$ is the fibration defined as
\[
\tau_m(\phi)=\left(\phi(0),\phi\left(\frac{1}{m-1}\right),\ldots, \phi\left(\frac{m-2}{m-1}\right),\phi(1)\right).
\] 

We will use the same notations in the above pullback diagrams when the space $X^m$ is replaced with any subspace $A_i\subset X^m$.

\begin{lemma}\label{extra1}
    Given a map $f:X\to Y$ and integer $m\ge 2$, if $\dTC_m(f)<n$, then there exist sets $A_1,A_2,\ldots,A_n$ that cover $X^m$ such that over each $A_i$ there is a section to the map $\theta_n$ from~\eqref{two22}.
\end{lemma}

\begin{proof}
Suppose that $\dTC_m(f)<n$. Then there exists a map $s_m:X^m\to \B_n(P(Y))$ as defined in Section~\ref{2c}. Proceeding as in~\cite[Lemma 4.3]{J1}, we define sets
    \[
    A_i=\left\{ \ov{x} = (x_1, \ldots, x_m) \in X^m \mid |\supp (s_m(\ov{x}))| = i\right\}
\]
for each $1\le i \le n$ and corresponding maps $H_i : A_i \to SP^{n!}(P(Y))$ such that
\[
H_i (\ov{x}) = \sum_{\phi \hspace{1mm} \in \hspace{1mm} \supp (s_m(\ov{x}))} \hspace{1mm} \frac{n!}{i} \phi.
\]
Clearly, $\{A_i\}_{i=1}^n$ is a cover of $X^m$ and $\zeta_n \circ H_i=\pa_m f^m$ for each $1 \le i \le n$, where we use $f^m$ to denote the restriction of $f^m$ to $A_i$. Then, $(r\circ \zeta_n)\circ H_i=\delta_{n!}^m \circ f^m$. So, by the universal property of the pullback on the right in~\eqref{two22}, there exists a map $K_i:A_i\to \mathcal{Z}_{n,m}$ such that $\theta_n \circ K_i$ is the identity map on $A_i$. 
\end{proof}

\begin{proof}[Proof of Theorem~\ref{dd}]
Let $\Delta^{n!}:SP^{n!}(Y)\to(SP^{n!}(Y))^m$ be the diagonal map. Define a homotopy equivalence $g:SP^{n!}(Y)\to SP^{n!}(P(Y))$ as 
    \[
    g\left[a_1, \ldots, a_{n!}\right] = \left[c_{a_1}, \ldots, c_{a_{n!}}\right],
    \]
    where $c_{a_j}$ denotes the trivial loop at $a_j\in Y$, see~\cite[Lemma 4.2]{J1}. 
If we assume that $\dTC_m(f)<n$, then by Lemma~\ref{extra1}, we get sets $A_i\subset X^m$ and maps $K_i:A_i\to \mathcal Z_{n,m}$ such that $\bigcup_{i=1}^nA_i=X^m$ and $\theta_n\circ K_i=\mathbbm{1}_{A_i}$. For each fixed $1\le i \le n$, we have the following commutative diagram in $R$-cohomology.
    \begin{equation}\label{three33}
    \begin{tikzcd}[every arrow/.append style={shift left}]
H^{k_i}(A_i;R)  \arrow{r}{\theta_n^*}
& 
H^{k_i}(\mathcal{Z}_{n,m};R) \arrow{l}{{K_i^* \vphantom{1}}}
\\
H^{k_i}((SP^{n!}(Y))^m;R) \arrow{r}{(r\zeta_n)^*} \arrow[swap]{u}{(\delta_{n!}^m f^m)^*} \arrow{d}{(\Delta^{n!})^*}
& 
H^{k_i}(SP^{n!}(P(Y));R) \arrow[swap]{u}{b^*} \arrow{dl}{g^*}
\\
H^{k_i}(SP^{n!}(Y);R) 
& 
\end{tikzcd}
    \end{equation}
By hypothesis, $(\Delta^{n!})^*(\alpha_i^*)=0$. Because $g^*$ is an isomorphism, $(r \circ \zeta_n)^*(\alpha_i^*)=0$. From the commutative square in~\eqref{three33}, we get that
\[
(\delta_{n!}^m \circ f^m)^*(\alpha_i^*)=(K_i^*\circ b^*)((r \circ \zeta_n)^*(\alpha_i^*))=0.
\]
Seeing the map $\delta_{n!}^m  \circ  f^m:A_i\to (SP^{n!}(Y))^m$ as the inclusion, we have the following long exact sequence of the pair $((SP^{n!}(Y))^m,A_i)$ in Alexander--Spanier cohomology:
\[
\cdots \to H^{k_i}((SP^{n!}(Y))^m,A_i; R) \xrightarrow{j_i^*} H^{k_i}((SP^{n!}(Y))^m;R) \xrightarrow{(\delta_n^m f^m)^*} H^{k_i}(A_i;R) \to \cdots,
\]
where $j_i:(SP^{n!}(Y))^m\hookrightarrow ((SP^{n!}(Y))^m,A_i)$. Since $(\delta_{n!}^m  \circ f^m)^*(\alpha_i^*)=0$, there exists a class $\ov{\alpha_i^*} \in H^{k_i}((SP^{n!}(Y))^m),A_i; R)$ such that $j_i^*(\ov{\alpha_i^*}) = \alpha_i^*$. Also, let $(\delta_{n!}^m  \circ f^m)^*(\ov{\alpha_i^*})=\ov{\alpha_i}\in H^{k_i}(X^m,A_i;R)$. For the maps $j=\sum j_i$ and $j'=\sum j_i'$, where $j_i':X^m\hookrightarrow (X^m,A_i)$, and integer $k = \sum k_i$, we get the following commutative diagram in Alexander--Spanier cohomology.
\begin{equation}\label{four44}
\begin{tikzcd}[contains/.style = {draw=none,"\in" description,sloped}]
H^{k}(X^m;R)  
& 
\arrow[swap]{l}{(j')^*} H^{k}\left(X^m,\bigcup_{i=1}^{n}A_i;R\right)  
\\
H^{k}\left((SP^{n!}(Y))^m;R\right)  \arrow[swap]{u}{(\delta^m_{n!} f^m)^{*}}
& 
H^k \left((SP^{n!}(Y))^m,\bigcup_{i=1}^n A_{i};R\right). \arrow[swap]{u}{(\delta^m_{n!} f^m)^{*}} \arrow[swap]{l}{j^*} 
\end{tikzcd}
\end{equation}
The cup product $\ov{\alpha_1^*} \smile \cdots \smile \ov{\alpha_n^*}$ in the bottom-right part of~\eqref{four44} goes to the non-zero cup product $\alpha_1 \smile \cdots \smile \alpha_n \neq 0$ in the top-left part part of~\eqref{four44}. But in the process, it factors through $\ov{\alpha_1} \smile \cdots \smile \ov{\alpha_n} \in H^k (X^m,X^m;R) = 0$. This is a contradiction. Therefore, we must have $\dTC_m(X)\ge n$.
\end{proof}

Basically, for the diagonal map $\Delta^{n!}:SP^{n!}(Y)\to(SP^{n!}(Y))^m$ and the map $\delta_{n!}^m \circ f^m:X^m\to(SP^{n!}(Y))^m$, we have for each fixed $m\ge 2$ and ring $R$ that
\[
\cu_R(\textup{Im}((\delta_{n!}^m \circ f^m)^*:\textup{Ker}((\Delta^{n!})^*)\to H^*(X^m;R))) \ge n \implies\dTC_m(f)\ge n.
\]

\begin{proof}[Proof of Corollary~\ref{dd1}]
Let $\szcl^m_{\Q}(X)=n$. Then there are $m$-th rational special zero-divisors of $X$, say $u_i\in H^*(X^m;\Q)$, such that $u_1\smile\cdots\smile u_n\ne 0$. By the proof of Theorem~\ref{cc}, for each $i$, there exists an $m$-th rational zero-divisor of $SP^{n!}(X)$, say $v_i\in H^*((SP^{n!}(X))^m;\Q)$, such that $(\delta_{n!}^m)^*(v_i)=u_i$. On finite CW complexes, singular cohomology and Alexander--Spanier cohomology coincide,~\cite{Sp}. Hence, we can apply Theorem~\ref{dd}, by taking $f$ to be the identity map on $X$ and $R=\Q$, with cohomology classes $v_i$ for $1\le i \le n$ to get $\dTC_m(X)\ge n=\szcl_{\Q}^m(X)$.
\end{proof}

\section{Symmetric products of orientable surfaces}\label{symorient}

Let $M_g$ denote the orientable surface of genus $g\ge 0$. In this section, we shall prove Theorem~\ref{aa}, which computes the $m$-th sequential topological complexity of the symmetric $n$-th product $SP^n(M_g)$ for each $m,n\ge 2$ and $g\ge 0$ using the $m$-th rational zero-divisors of $SP^n(M_g)$.

From now onwards, for brevity, we skip the symbol ``\hspace{0.5mm}$\smile$\hspace{0.5mm}'' while writing cup products, i.e., we let $xy:=x\smile y$.

\subsection{Cohomology}\label{cohomg} Throughout this section, we let $H^*(X):=H^*(X;\Z)$ and $H_*(X):=H_*(X;\Z)$. 

The CW structure on the orientable surface $M_g$ consists of one $0$-cell, $2g$ $1$-cells, and one $2$-cell attached by the product of commutators $[a_1,b_1]\cdots[a_g,b_g]$. Here, $a_i,b_i\in H_1(\T_i)$ are the generators for each $i$ (and $\T_i$ is a $2$-dimensional torus). Let us write $J:=\prod_{i=1}^g \T_i$ for the Jacobian of $M_g$, and $\text{proj}_i:J\to \T_i$ for the projection map for each $i$. Let $\mu_n:SP^n(M_g)\to J$ be the Abel--Jacobi map (see, for instance,~\cite[Sections 2.C \& 2.D]{DDJ} for its description).

Let $a_i^*$ and $b_i^*$ denote the Hom duals of $a_i$ and $b_i$, respectively. We use the same notations for their images in $H^1(SP^n(M_g))$ under the induced map $(\text{proj}_i \circ \mu_n)^*$. Let $c\in H_2(M_g)$ denote the fundamental class and its image in $H_2(SP^n(M_g))$ under the map induced in homology by the basepoint inclusion $M_g\hookrightarrow SP^n(M_g)$. We denote the Hom dual of $c\in H_2(SP^n(M_g))$ by $c^*$. Note that in $H^*(M_g)$, we have $a_i^*b_j^*=a_i^*a_j^*=b_i^*b_j^*=0$ for each $i\ne j$ and $(a_i^*)^2=(b_i^*)^2=0$,~\cite{Ha}. Also, $c^*$ commutes with $a_i^*$ and $b_j^*$, and $a_i^*$ and $b_j^*$ anti-commute with each other in $H^*(SP^n(M_g))$.

The following famous theorem is due to Macdonald,~\cite{Mac}.

\begin{thm}\label{mc}
    The integral cohomology ring $H^*(SP^n(M_g))$ is the quotient of $\Lambda_{\Z}(a_1^*,b_1^*,\dots,a_g^*,b_g^*)\otimes \Z[c^*]$ by the following relation:
    \[
    a_{i_1}^*\cdots a_{i_\ell}^*b_{j_1}^*\cdots b_{j_m}^*(c^*-a_{k_1}^*b_{k_1}^*)\cdots(c^*-a_{k_r}^*b_{k_r}^*)(c^*)^s=0
    \]
    whenever $\ell+m+2r+s\ge n+1$ for any collection of distinct indices $i_1,\dots,i_\ell$, $j_1,\dots,j_m$, and $k_1,\dots, k_r$.
\end{thm}

The following statement was proven in~\cite[Proposition 3.3]{DDJ} using the above description. 

\begin{prop}\label{cup1}
In the integral cohomology ring $H^*(SP^n(M_g))$,
\vspace{-2.0mm}
\begin{enumerate}
        \itemsep-0.30em 
\item if $n\le g$, then the cup product $a_1^*b_1^*\cdots a_n^*b_n^*$ is non-zero, and
\item if $n>g$, then the cup product $a_1^*b_1^*\cdots a_g^*b_g^*(c^*)^{n-g}$ is non-zero.
\end{enumerate}
\end{prop}

The integral cohomology ring $H^*(SP^n(M_g))$ is torsion-free,~\cite{Mac},~\cite{KS}. So, $a_i^*$, $b_j^*$, $c^*\in H^*(SP^n(M_g))$ can be considered as rational cohomology classes and hence, Theorem~\ref{mc} and Proposition~\ref{cup1} hold true in rational cohomology as well.

The LS-category of $SP^n(M_g)$ can now be computed easily.

\begin{thm}[\protect{\cite[Theorem 6.5]{DDJ}}]\label{oldknown}
    For each $n\ge 1$ and $g\ge 0$, we have that
    \[
    \cat\left(SP^n(M_g)\right) = 
    \begin{cases}
    2n & \text{if } n\le g\\
    n+g & \text{if } n>g.
    \end{cases}
    \]
\end{thm}

\begin{proof}
    For $n\le g$, we have due to Proposition~\ref{cup1} (1) and Theorem~\ref{clbound} that 
    \[
    2n\le \cu_{\Q}(SP^n(M_g))\le \cat(SP^n(M_g))\le \dim(SP^n(M_g))=2n.
    \]
    Similarly, for $n>g$, Proposition~\ref{cup1} (2) and Theorem~\ref{clbound} imply the inequality $n+g\le \cu_{\Q}(SP^n(M_g))\le \cat(SP^n(M_g))$. In this case, the main result of~\cite{Dr3} gives
    \[
\cat(SP^n(M_g))\le\frac{\dim(SP^n(M_g))+\cd(\pi_1(SP^n(M_g)))}{2}=\frac{2n+2g}{2}=n+g.
    \]
    This completes the proof.
\end{proof}

\subsection{Sequential topological complexity}\label{tcm}

For $m\ge 2$, let $\Delta_m:X \to X^m$ denote the diagonal map. For each $1 \le i \le m$, let $\pi_i:X^m\to X$ be the projection onto the $i$-th factor. For any $y\in H^*(X;\Q)$, let us define $\ov{y}\in H^*(X^m;\Q)$ as
\[
\ov{y}=\sum_{i=1}^{m-1}\pi_i^*(y)-(m-1)\pi_m^*(y).
\]
Here, $\pi_i^*(y)=1\otimes\cdots\otimes 1\otimes y\otimes 1 \otimes \cdots\otimes 1\in H^*(X^m;\Q)$ for each $1\le i\le m$, where $y$ appears in the $i$-th position. We note that $\Delta_m^*(\ov{y})=(m-1)y-(m-1)y=0$, see~\cite[Section 4]{Ru}. Hence, $\ov{y}$ is an $m$-th rational special zero-divisor of $X$, where ``special'' is in the sense of Definition~\ref{szcldefn}.

We will follow these notations in the proof of our main result.

\begin{proof}[Proof of Theorem~\ref{aa}]
    Let us fix $m\ge 2$. For each $n\ge 2$ and $g\ge 0$, we have the upper bound $\TC_m(SP^n(M_g))\le m\cat(SP^n(M_g))$ due to Theorem~\ref{tcbound}, where the value $\cat(SP^n(M_g))$ is known by Theorem~\ref{oldknown}. Let us now find the lower bounds separately in the cases $n\le g$ and $n>g$.

    We note that the classes $\ov{a_i^*},\ov{b_j^*},\ov{c^*}\in H^*((SP^n(M_g))^m;\Q)$ are non-zero. We also recall that $\ov{a_i^*},\ov{b_j^*},\ov{c^*}\in\text{Ker}(\Delta_m^*)$ for each $i,j$. For each $k\ge 2$, $(a_i^*)^k=(b_j^*)^k=0$. Hence, it follows from the relationship between the cup product and the tensor product~\cite[Section 3.2]{Ha} that for each $1\le i,j\le g$,
    \[
    (\ov{a_i^*})^m=\pm (m-1)m! \underbrace{\left(a_i^*\otimes\cdots\otimes a_i^*\right)}_{\text{$m$-times}}  \hspace{3.5mm} \text{and} \hspace{3.5mm} (\ov{b_j^*})^m=\pm (m-1)m! \underbrace{\left(b_j^*\otimes\cdots\otimes b_j^*\right)}_{\text{$m$-times}}.
    \]
Let $n\le g$. Then $a_1^*b_1^*\cdots a_n^*b_n^*\ne 0$ by Proposition~\ref{cup1} (1). Hence, the following cup product is non-zero:
    \[
    (\ov{a_1^*})^m (\ov{b_1^*})^m\cdots (\ov{a_n^*})^m  (\ov{b_n^*})^m= \pm \left((m-1)m!\right)^{2n} \underbrace{\left(\left(a_1^*b_1^*\cdots a_n^*b_n^*\right)\otimes\cdots\otimes \left(a_1^*b_1^*\cdots a_n^*b_n^*\right)\right)}_{\text{$m$-times}}. 
    \]
    Therefore, we get $2mn\le \zcl_{\Q}^m(SP^n(M_g))\le \TC_m(SP^n(M_g))$ from Theorem~\ref{tcbound}. Thus, we have that $\TC_m(SP^n(M_g))=2mn$ in the case $n\le g$.

    Next, we let $n> g$. It is not difficult to see that $(\ov{c^*})^{m(n-g)}$ contains the term
    \[
    (-1)^{n-g}(m-1)^{n-g}\prod_{i=0}^{m-1} {(m-i)(n-g) \choose n-g} \underbrace{\left((c^*)^{n-g}\otimes\cdots\otimes(c^*)^{n-g}\right)}_{\text{$m$-times}}.
    \]
    By Proposition~\ref{cup1} (2), $a_1^*b_1^*\cdots a_g^*b_g^*(c^*)^{n-g}\ne 0$. Hence, the above term and the following term are both non-zero.
    \[
    (\ov{a_1^*})^m (\ov{b_1^*})^m\cdots (\ov{a_g^*})^m  (\ov{b_g^*})^m= \pm\left((m-1)m!\right)^{2g}\underbrace{\left(\left(a_1^*b_1^*\cdots a_g^*b_g^*\right)\otimes\cdots\otimes \left(a_1^*b_1^*\cdots a_g^*b_g^*\right)\right)}_{\text{$m$-times}}. 
    \] 
    Also, for each $k \ge 1$, it follows from Theorem~\ref{mc} that $a_1^*b_1^*\cdots a_g^*b_g^*(c^*)^{n-g+k}=0$. Thus,
    we finally have that the cup product $(\ov{a_1^*})^m (\ov{b_1^*})^m\cdots (\ov{a_g^*})^m  (\ov{b_g^*})^m (\ov{c^*})^{m(n-g)}$ is equal to the non-zero term
    \[
     \underbrace{\left(\left(a_1^*b_1^*\cdots a_g^*b_g^*(c^*)^{n-g}\right)\otimes\cdots\otimes\left(a_1^*b_1^*\cdots a_g^*b_g^*(c^*)^{n-g}\right)\right)}_{\text{$m$-times}}
    \]
    multiplied by the non-zero constant
    \[
    \pm(-1)^{n-g}(m-1)^{n+g}(m!)^{2g}\prod_{i=0}^{m-1} {(m-i)(n-g) \choose n-g}.
    \]
    Therefore, $2mg+m(n-g)=m(n+g)\le \zcl_{\Q}^m(SP^n(M_g))\le \TC_m(SP^n(M_g))$ due to Theorem~\ref{tcbound}. Hence, $\TC_m(SP^n(M_g))=mn+mg$ in the case $n> g$.
\end{proof}

For completeness, we recall that for any integer $m\ge 2$, $\TC_m(M_0)=m$, $\TC_m(M_1)=2m-2$, and $\TC_m(M_g)=2m$ for each $g\ge 2$. 

We now verify the rationality conjecture (Conjecture~\ref{foconjec}) for $SP^n(M_g)$ for each $n\ge 2$ and $g\ge 0$. We note that this conjecture has been verified for only a few classes of spaces, namely the spheres, simply connected symplectic manifolds, compact Lie groups, surfaces\hspace{0.2mm}\footnote{\hspace{0.5mm}The verification of the rationality conjecture for non-orientable surfaces seems to be missing from~\cite{FO} and the literature thereafter. Indeed, Conjecture~\ref{foconjec} holds for $N_g$ in view of the computations of~\cite{GGGL},~\cite{Dr2}, and~\cite{CV}. The rational function represented by the power series of the $\TC$-generating function of $N_g$ in the case $g=1$ is different from that in the case $g\ge 2$.}, and the classifying spaces of Higman's group, Right-angled Artin groups, and torsion-free hyperbolic groups, see~\cite{FO} and~\cite{HL}. 

\begin{proof}[Proof of Corollary~\ref{aa1}]
Using Theorem~\ref{aa}, we see that the $\TC$-generating function of $X=SP^n(M_g)$ is the power series
   \[
   f_{SP^n(M_g)}(t)=\sum_{m=1}^{\infty} (m+1)\cat(SP^n(M_g))\hspace{0.5mm}t^m 
   = \cat(SP^n(M_g))\hspace{0.5mm} \frac{d}{dt}\left(\sum_{m=1}^{\infty} t^{m+1}\right),
   \]
   which represents the rational function
   \[
   h_{SP^n(M_g)}(t)=\frac{(2t-t^2)\cat(SP^n(M_g))}{(1-t)^2}.
   \]
   Clearly, the numerator of $h_{SP^n(M_g)}$ is an integral quadratic polynomial whose value at $t=1$ is $\cat(SP^n(M_g))$.
\end{proof}

We now show that the sequential distributional complexities of $SP^n(M_g)$ and the distributional category of its products agree with their classical counterparts. 

\begin{proof}[Proof of Corollary~\ref{dd2}]
    For each $m\ge 2$, the proof of Theorem~\ref{aa} gives the equality $\TC_m(SP^n(M_g))=\szcl_{\Q}^m(SP^n(M_g))$. Hence, Corollary~\ref{dd1} and Theorem~\ref{dtcbound} imply that $\dTC_m(SP^n(M_g))=\TC_m(SP^n(M_g))$. Using this and Corollary~\ref{aa2}, we get from Theorems~\ref{dcatbound} and~\ref{dtcbound} that $\dcat((SP^n(M_g))^m)=\cat((SP^n(M_g))^m)$.
\end{proof}

\begin{remark}\label{obvobv}
    It is well-known that $SP^n(M_0)=\C P^n$, see, for example,~\cite{KS}. We observe that Theorem~\ref{aa} and Corollary~\ref{aa1} recover the corresponding results for $\C P^n$. From Corollary~\ref{dd2}, we also deduce using~\cite[Corollary 3.15]{BGRT} that $\dTC_m(\C P^n)=mn$, a computation not found in~\cite{J1}.
\end{remark}

\subsection{More consequences}\label{observe}

Let us highlight some ways in which the LS-category and the sequential topological complexity of $SP^n(M_g)$ are ``nicely behaved''.

\begin{enumerate}[(1)]

\item \textit{Logarithmic law conjecture:} For spaces $Y_i$,~\cite[Proposition 2.3]{Ja} implies that
    \begin{equation}\label{summation}
    \cat\left(Y_1\times Y_2\times\cdots\times Y_m\right)\le \sum_{i=1}^m \cat\left(Y_i\right)
    \end{equation}
 for any $m\ge 1$. This leads to the ``logarithmic law conjecture'', which says that for any closed manifold $Y$, $\cat(Y^m)=m\cat(Y)$. This conjecture is not true in general, see~\cite{Dr1} and the references therein. However, for $Y=SP^n(M_g)$ and $Y=SP^n(N_g)$, this conjecture is verified for each valid $m,n$, and $g$ in view of Corollaries~\ref{aa2} and~\ref{bb2}, respectively. 

\item \emph{The Ganea conjecture:} For any $k\ge 1$ and space $Y$, $\cat(Y \times S^k)\le\cat(Y)+1$. Ganea's conjecture~\cite{CLOT} says that the above inequality should be equality, in particular, if $Y$ is a finite CW complex. This conjecture was disproved by Iwase in~\cite{Iwmain} (see also~\cite{Iw}). Interestingly, we can verify this conjecture for $Y=SP^n(M_g)$ for each $n$, $g$, and $k$. Using techniques from the proof of Theorem~\ref{oldknown} and~\eqref{summation}, it can be shown that
\[
\cat(SP^n(M_g)\times S^k)=\begin{cases}
    2n+1 & \text{ if } n\le g\\
    n+g+1 & \text{ if } n>g.
\end{cases}
\]

\item \emph{The $\TC_m$-Ganea conjecture:} For sequential topological complexities, one can pose an analogue of Ganea's original conjecture on LS-category by asserting that $\TC_m(Y\times S^k)=\TC_m(Y)+\TC_m(S^k)$ for any finite CW complex $Y$ and integers $m\ge 2$ and $k\ge 1$. For $m=2$, this conjecture was disproved in~\cite{GGV} in the case when $k$ is even. However, this conjecture is open in the case of odd $k$, even for $m=2$ (see, for example,~\cite{GV}). Here, we can verify this conjecture for $Y=SP^n(M_g)$ for each $n$, $g$, $m$, and $k$. Indeed,~\cite[Theorem~3.10 \& Proposition 3.11]{BGRT} give 
$\zcl_{\Q}^m(SP^n(M_g))+\zcl_{\Q}^m(S^k)\le\zcl_{\Q}^m(SP^n(M_g)\times S^k)$ and $\TC_m(SP^n(M_g)\times S^k)\le\TC_m(SP^n(M_g))+\TC_m(S^k)$, respectively, and we know from~\cite{Ru} and the proof of Theorem~\ref{aa}, respectively, that $\zcl_{\Q}^m(S^k)=\TC_m(S^k)$ and $\zcl_{\Q}^m(SP^n(M_g))=\TC_m(SP^n(M_g))$.

   \item \textit{Growth of the sequence $\{\TC_m(X)\}_{m\ge 2}$:} For a space $X$ and $m\ge 2$, we have from~\cite[Proposition 3.7]{BGRT} that $\TC_{m+1}(X)-\TC_m(X)\le 2\cat(X)$. Here, for each $m\ge k$, 
   \[
   \TC_m(SP^n(M_g))-\TC_k(SP^n(M_g))=(m-k)\cat(SP^n(M_g))
   \]
   due to Theorem~\ref{aa}. In particular, the sequence $\{\TC_{m}(SP^n(M_g))\}_{m\ge 2}$ is an arithmetic progression with common difference $\cat(SP^n(M_g))$. 
\end{enumerate}

Let us now make more observations using Theorems~\ref{aa} and~\ref{oldknown}.

\begin{enumerate}[(i)]
\item\label{no relation} Let $G$ be a compact Lie group acting on a finite CW complex $X$. Let $\cat_G(X)$ and $\TC_G(X)$ denote the $G$-equivariant LS-category~\cite{Fad},~\cite{Mar} and the $G$-equivariant topological complexity~\cite{CG} of $X$, respectively. While we have\footnote{\hspace{0.2mm}Strictly speaking, we need to assume that the $G$-space $X$ is $G$-connected (i.e., the fixed-point set $X^H$ is path-connected for every closed subgroup $H\subset G$) and $X^G\ne \emptyset$ for the inequality $\cat(X)\le\cat_G(X)$. The author thanks Navnath Daundkar for pointing this out.}
\[
\max\{\cat(X/G),\cat(X)\}\le\cat_G(X) \hspace{3.5mm} \text{and} \hspace{3mm} \max\{\TC(X/G),\TC(X)\}\le\TC_G(X),
\]
there is no relation between $\cat(X/G)$ and $\cat(X)$, and between $\TC_m(X/G)$ and $\TC_m(X)$ for any $m\ge 2$. To see this, first, take $G=\Z_2$ and $X=S^n$ for $n\ge 2$. Note that $X/G=\R P^n$. Clearly, $\cat(X)= 1 < n=\cat(X/G)$ and $\TC_m(X)\le m < \TC_m(X/G)$. Next, let us take $G=S_n$ and $X=(M_g)^n$ for $n>g\ge 2$. Note that $X/G=SP^n(M_g)$. By Theorems~\ref{oldknown} and~\ref{aa}, we get $\cat(X/G)= n+g < 2n=\cat(X)$ and $\TC_m(X/G)= mn+mg <\TC_m(X)$. 

\item It follows from~\cite[Proposition 3.11]{BGRT} that $\TC_m(X^n)\le n\TC_m(X)$ for any $m\ge 2$ and $n\ge 1$. However, it turns out that in general, there is no relation between $\TC_m(SP^n(X))$ and $n\TC_m(X)$. For each $n>g\ge 2$, the inequality $\TC_m(SP^n(M_g))<n\TC_m(M_g)$ follows from the above item. Using Theorem~\ref{aa}, we also see that $\TC(SP^n(\T))>n\TC(\T)\ge \TC(\T^n)$ for each $n\ge 2$ and that $\TC_3(SP^2(\T))>2\TC_3(\T)\ge \TC_3(\T^2)$ for the $2$-torus $\T$.

\item It is well-known that $SP^2(M_1)$ is a $S^2$-bundle over a $2$-torus, say $\T$. Let $p:SP^2(M_1)\to \T$ be the corresponding locally trivial fibration with fiber $S^2$. Then for the $m$-th sequential parametrized topological complexity of the fibration $p$ (see~\cite{FP}, and also~\cite{CFW} for the case $m=2$), we have because of~\cite[Definition 3.1 and Proposition 6.1]{FP} that 
\[
m=\TC_m\left(S^2\right)\le\TC_m\left[\hspace{0.4mm}p:SP^2\left(M_1\right)\to \T\right]\le\frac{2m+2+1}{2} = m+\frac{3}{2}.
\]
However, $\TC_m(SP^2(M_1))=3m$ by Theorem~\ref{aa}. Hence, for each $m\ge 2$, $\TC_m$ of the total space of $p$ is almost thrice the value of $\TC_m[\hspace{0.4mm}p:SP^2(M_1)\to\T]$.
\end{enumerate}

\subsection{Products of symmetric products of orientable surfaces}\label{products section}
We now study $\cat$ and $\TC_m$ of finite products of the symmetric products of closed orientable surfaces. We begin by recalling the following definition from ~\cite[Definition 1.3]{AGO}.

\begin{defn}[\protect{\cite{AGO}}]\label{tclogdefn}
A family of spaces $\{X_i\}_{i\in I}$ is said to be \emph{LS-logarithmic} if $\cat(\prod_{i\in J}X_{i_j})=\sum_{j\in J}\cat(X_{i_j})$ holds for each finite subset $J\subset I$. Similarly, a family of spaces $\{Y_i\}_{i\in I}$ is called \emph{$\TC_m$-logarithmic} for an integer $m\ge 2$ if $\TC_m(\prod_{i\in J}Y_{i_j})=\sum_{j\in J}\TC_m(Y_{i_j})$ holds for each finite $J\subset I$.
\end{defn}

In particular, an LS-logarithmic family with $X_i=X_k$ for all $i,k\in X$ satisfies the logarithmic law conjecture (\emph{cf.} Section~\ref{observe}).

\begin{proof}[Proof of Corollary~\ref{aa2}]
First, note that the proof of Theorem~\ref{oldknown} gives the equality $\cu_{\Q}(SP^n(M_g))=\cat(SP^n(M_g))$ for each $n\ge 1$ and $g\ge 0$. The K\"unneth formula gives an isomorphism of $\Q$-algebras (and hence, of $\Q$-vector spaces)
\[
H^*(SP^{n_1}(M_{g_1})\times SP^{n_2}(M_{g_2});\Q)=H^*(SP^{n_1}(M_{g_1});\Q)\otimes_{\Q} H^*(SP^{n_2}(M_{g_2});\Q),
\]
which implies $\cu_{\Q}(SP^{n_1}(M_{g_1}))+\cu_{\Q}(SP^{n_2}(M_{g_2}))\le \cu_{\Q}(SP^{n_1}(M_{g_1})\times SP^{n_2}(M_{g_2}))$. Using an inductive argument, the LS-logarithmicity of $\{SP^n(M_g)\}_{n,g}$ follows in view of~\eqref{summation}. Such an argument also yields $\cat((SP^n(M_g))^m)=m\cat(SP^n(M_g))$. We note that the latter also follows directly from Theorems~\ref{aa} and~\ref{tcbound}. 
\end{proof}

We now look at the $\TC_m$-logarithmicity of $\{SP^n(M_g)\}_{n\ge 1,g\ge 0}$. The following lemma will be crucial. 

\begin{lemma}\label{important lemma}
For any integers $m, n_1,n_2\ge 2$, and $g_1,g_2\ge 0$, we have that 
\[
\szcl^m_{\Q}(SP^{n_1}(M_{g_1}))+\szcl^m_{\Q}(SP^{n_2}(M_{g_2}))\le\szcl^m_{\Q}(SP^{n_1}(M_{g_1})\times SP^{n_2}(M_{g_2})).
\]
\end{lemma}

\begin{proof}
For brevity, let $Y=SP^{n_1}(M_{g_1})$ and $Z=SP^{n_2}(M_{g_2})$, and let us write $H^*(A):=H^*(A;\Q)$ for any space $A$. For each $1\le i \le m$, define the projections $\pi^Y_i:Y^m\to Y$ and $\pi^Z_i:Z^m\to Z$ onto the respective $i$-th factors. Suppose $r_i:(Y\times Z)^m \to Y\times Z$ is the projection onto the $i$-th factor. Up to homeomorphism, $r_i=\pi_i^Y\times\pi_i^Z$. For non-zero elements $y\in H^*(Y)$ and $z\in H^*(Z)$, we have the special zero-divisors $\ov{y}\in H^*(Y^m)$ and $\ov{z}\in H^*(Z^m)$, defined, respectively, as
\[
\ov{y}=\sum_{i=1}^{m-1}(\pi_i^Y)^*(y)-(m-1)(\pi^Y_m)^*(y) \ \text{ and } \ \ov{z}=\sum_{i=1}^{m-1}(\pi_i^Z)^*(z)-(m-1)(\pi^Z_m)^*(z).
\]
Let $\pr_Y:Y\times Z\to Y$ and $\pr_Z:Y\times Z\to Z$ be the projections. We then have the product maps $\pr^m_Y:(Y\times Z)^m\to Y^m$ and $\pr^m_Z:(Y\times Z)^m\to Z^m$.
Consider $\wh{y}=\pr_Y^*(y)=y\otimes 1$ and $\wh{z}=\pr_Z^*(z)=1\otimes z$ in $H^*(Y\times Z)=H^*(Y)\otimes H^*(Z)$, and the cohomology classes $\wt{y}=(\pr^m_Y)^*(\ov{y})=\ov{y}\otimes 1$ and $\wt{z}=(\pr^m_Z)^*(\ov{z})=1\otimes\ov{z}$ in $H^*((Y\times Z)^m)=H^*(Y^m)\otimes H^*(Z^m)$, where the ``$=$'' symbol denotes the K\"unneth isomorphisms between the respective $\Q$-algebras. These isomorphisms imply that 
\[
\wh{y}\hspace{0.4mm}\wh{z}=\pm \hspace{0.4mm}y\otimes z\ne 0\ \text{ and } \ \wt{y}\hspace{0.4mm}\wt{z}=\pm\hspace{0.4mm}\ov{y}\otimes\ov{z}\ne 0
\]
since $\ov{y},\ov{z}\ne 0$.
Let $\Delta_m^Y:Y\to Y^m$ and $\Delta_m^Z:Z\to Z^m$ be the diagonal maps. Then, up to homeomorphism, $\Delta_m^Y\times\Delta_m^Z:Y\times Z\to (Y\times Z)^m$ is the diagonal map. Since $(\Delta_m^Y)^*(\ov{y})=(m-1)y-(m-1)y=0$, we have the following diagram for $Y$, which commutes up to isomorphism because of the properties of tensor and cup products.
\[
\begin{tikzcd}[column sep = huge, contains/.style = {draw=none,"\in" description,sloped}]
\bigoplus_{i=1}^{m-1} \wh{y} \oplus (m-1)\wh{y}\ar[r,mapsto] \ar[d,contains] & \wt{y} \ar[d,contains] \ar[r,mapsto] & (m-1)\wh{y}-(m-1)\wh{y} \ar[d,contains]
\\
\bigoplus_{i=1}^{m} H^*(Y\times Z) \arrow{r}{\bigoplus_{i=1}^{m-1} r_i^* - r_m^*} 
&
H^*((Y\times Z)^m)  \arrow{r}{(\Delta_m^Y\times\Delta_m^Z)^*}  
&
H^{*}(Y\times Z)
\\
\bigoplus_{i=1}^m H^*(Y)  \arrow[swap]{r}{\bigoplus_{i=1}^{m-1} (\pi_i^Y)^* - (\pi_m^Y)^*} \arrow{u}{\bigoplus_{i=1}^{m} \pr_Y^* \hspace{1mm}}
&
H^{*}(Y^m)  \arrow[swap]{r}{(\Delta_m^Y)^*}  \arrow{u}{(\pr^m_Y)^*}
&
H^{*}(Y). \arrow[swap]{u}{\pr_Y^*}
\\
\bigoplus_{i=1}^{m-1} y \oplus (m-1)y \ar[r,mapsto] \ar[u,contains] & \ov{y} \ar[u,contains] \ar[r,mapsto] & (m-1)y-(m-1)y \ar[u,contains]
\end{tikzcd}
\]
From this, we get that $\wt{y}$ is an $m$-th rational special zero-divisor of $Y\times Z$, in the sense of Definition~\ref{szcldefn}. Using a similar argument, we deduce that $\wt{z}$ is also such a special zero-divisor of $Y\times Z$. Now, recall from the proof of Theorem~\ref{aa} that $\szcl^m_{\Q}(Y)$ and $\szcl^m_{\Q}(Z)$ are determined precisely by the special zero-divisors of the form $\ov{y}$ and $\ov{z}$, respectively. They produce special zero-divisors of $Y\times Z$, namely $\wt{y}$ and $\wt{z}$, whose cup product is non-zero. Therefore, it follows from the proof of Theorem~\ref{aa} that $\szcl^m_{\Q}(Y\times Z)\ge \szcl^m_{\Q}(Y)+\szcl^m_{\Q}(Z)$. 
\end{proof}

\begin{proof}[Proof of Corollary~\ref{aa3}]
    By the proof of Theorem~\ref{aa}, we know that the equality $\szcl^m_{\Q}(SP^n(M_g))=\TC_m(SP^n(M_g))$ holds for each $m,n\ge 2$ and $g\ge 0$. From this, Lemma~\ref{important lemma}, and~\cite[Proposition 3.11]{BGRT}, we conclude that
    \[
    \TC_m(SP^{n_1}(M_{g_1})\times SP^{n_2}(M_{g_2})) = \TC_m(SP^{n_1}(M_{g_1}))+\TC_m(SP^{n_2}(M_{g_2})).
    \]
    If $n_1=n_2$ and $g_1=g_2$, we obtain $\TC_m((SP^n(M_g))^2)=2\TC_m(SP^n(M_g))$. Hence, the $\TC_m$-logarithmicity of the family, as well as the second assertion, both follow from an inductive argument on Lemma~\ref{important lemma}.
\end{proof}

Note that we have $\TC_m(\prod_{i=1}^kSP^{n_i}(M_{g_i}))=m\cat(\prod_{i=1}^kSP^{n_i}(M_{g_i}))$ for all $k\ge 1$, $m,n_i\ge 2$, and $g_i\ge 0$ due to Theorem~\ref{aa} and Corollaries~\ref{aa2} and~\ref{aa3}. Hence, the rationality conjecture (Conjecture~\ref{foconjec}) is verified also for finite products of the symmetric products of closed orientable surfaces (\emph{cf.} proof of Corollary~\ref{aa1}).

\section{Symmetric products of non-orientable surfaces}\label{symnonorient}

Let $N_g$ denote the non-orientable surfaces of genus $g\ge 1$. In this section, we shall prove Theorem~\ref{bb}, which computes the Lusternik--Schnirelmann category of the symmetric $n$-th product $SP^n(N_g)$ for each $n,g\ge 1$. We will also show that the $i$-th higher homotopy group of $SP^n(N_g)$ for $n>i\ge 2$ is isomorphic to the $i$-th reduced integral homology group of $N_g$, which vanishes (\emph{c.f.} Theorem~\ref{ee}).

\subsection{Homotopy groups}

It follows from the Dold--Thom Theorem~\cite{DT} that $SP^{\infty}(M_g)\simeq (S^1)^{2g}\times\C P^{\infty}$ and $SP^{\infty}(N_g)\simeq (S^1)^{g-1}\times\R P^{\infty}$ for each $g$. In fact, these homotopy equivalences are recovered from the following general result on homotopy splitting (see~\cite[Section 6]{KS} for a proof).

\begin{prop}\label{2dim}
    If $X$ is a $2$-dimensional CW complex     such that 
    \[
    H_1(X;\Z)= \Z_{k_1}\oplus\cdots\oplus\Z_{k_r}\oplus\Z^a \hspace{3mm} \text{and} \hspace{3mm} H_2(X;\Z)=\Z^b
    \]
    for some non-negative integers $a,b,r$, and $k_i$, then 
    \[
    SP^{\infty}(X)\simeq (S^1)^a\times \left(\C P^{\infty}\right)^b\times L_{k_1}\times\cdots\times L_{k_r},
    \]
    where $L_{k_i}=S^{\infty}/\Z_{k_i}$ is the classifying space of the group $\Z_{k_i}$ for each $i$.
\end{prop}

We use the above proposition and our discussion in Section~\ref{symprod} to obtain the following result on the vanishing of certain higher homotopy groups of $SP^n(X)$.

\begin{prop}\label{hh}
    Let $X$ be a $2$-dimensional CW complex with only one vertex, satisfying the hypothesis of Proposition~\ref{2dim}. Then, for all $n\ge 3$, $\pi_2(SP^n(X))=\Z^b$, and for each $i\ge 3$, $\pi_i(SP^n(X))=0$ for all $n> i$.
    \end{prop}

\begin{proof}
Let us fix some $i\ge 2$. Then for any given $n\ge i+1$, we get from~\eqref{skeleton} that the $(i+1)$-skeleton of $\ov{SP}^{\infty}(X)$ coincides with the $(i+1)$-skeleton of $\ov{SP}^n(X)$. In view of the cellular approximation theorem~\cite{Ha}, the basepoint inclusion map $\ov{SP}^n(X)\hookrightarrow\ov{SP}^{\infty}(X)$ can be assumed to be cellular. 
Thus, for $n>i\ge 2$, we get
    \[
\pi_i(SP^n(X))=\pi_i(\ov{SP}^n(X))=\pi_i(\ov{SP}^{\infty}(X))=\pi_i(SP^{\infty}(X)).
    \]
From Proposition~\ref{2dim}, we have that $\wt{SP^{\infty}(X)}\simeq \R^a \times (\C P^{\infty})^b\times (S^{\infty})^r\simeq (\C P^{\infty})^b$. The conclusion now follows directly from the fact that $\pi_2((\C P^{\infty})^b)=\Z^b$ and $\pi_i((\C P^{\infty})^b)=0$ for all $i \ge 3$,~\cite[Example 4.50]{Ha}.
\end{proof}

We now focus on the symmetric products of surfaces and prove Theorem~\ref{ee}.

\begin{proof}[Proof of Theorem~\ref{ee}]
    Clearly, $X=M_g$ satisfies the hypothesis of Proposition~\ref{2dim} with $a=2g$, $b=1$, and $r=0$. Hence, we apply Proposition~\ref{hh} to $X=M_g$ and conclude that the second homotopy group of $SP^n(M_g)$ is $\Z=\wt{H}_2(M_g;\Z)$ whenever $n\ge 3$, and $\pi_i(SP^n(M_g))=0=\wt{H}_i(M_g;\Z)$ for all $3\le i < n$.
    
    Similarly, $X=N_g$ satisfies the hypothesis of Proposition~\ref{2dim} with $r=1$, $k_1=2$, $a=g-1$, and $b=0$. So, we again apply Proposition~\ref{hh} to $X=N_g$ and conclude that $\pi_i(SP^n(N_g))=0=\wt{H}_i(N_g;\Z)$ for all $2\le i<n$.
\end{proof}

The existence of spin structures on the universal covers of $SP^n(N_g)$ for $n\ge 3$ can now be justified.

\begin{proof}[Proof of Corollary~\ref{ee1}]
We recall that an orientable Riemannian manifold $M$ has a spin structure if and only if its second Stiefel--Whitney class $w_2(M)\in H^2(M;\Z_2)$ vanishes. Since $\pi_2(SP^n(N_g))=0$ for $n\ge 3$ and $g\ge 1$ by Theorem~\ref{ee}, the universal cover $\wt{SP^n(N_g)}$ is $2$-connected and so its second Stiefel--Whitney class must vanish. This completes the proof.
\end{proof}

Motivated by~\cite[Theorem 9.12]{DDJ} in the case of the symmetric products of orientable surfaces, we ask the following question.

\begin{question}
    Is the orientable double cover of $SP^n(N_g)$ a spin $2n$-manifold for some integers $n\ge 2$ and $g\ge 2$?
\end{question}

\subsection{Cohomology}\label{cohong}

From this point onward, we shall write $H^*(X):=H^*(X;\Z_2)$, $H_*(X):=H_*(X;\Z_2)$, and  $H^*(X):=H^*(X;\Z_2)$. 

The structure of the $\Z_2$-cohomology ring $H^*(SP^n(N_g))$ in the case $n\ge 2$ was first studied by Kallel and Salvatore in~\cite[Section 4.2]{KS}. 

The CW structure on $N_g$ consists of one $0$-cell, $g$ $1$-cells, and one $2$-cell attached by the word $e_1^2e_2^2\cdots e_g^2$. Here, $e_i\in H_1(\P_i)$ is the generator for each $i$ (and $\P_i$ is a real projective plane). Using this $\Delta$-complex structure, the elements $e_i$ form a basis for $H_1(N_g)$. We use the same notations $e_i$ for their images in $H_1(SP^n(N_g))$ under the map induced in homology by the basepoint inclusion $N_g \hookrightarrow SP^n(N_g)$. Let $d\in H_2(N_g)$ denote the $\Z_2$-fundamental class and its image in $H_2(SP^n(N_g))$. Let us denote the Hom duals of $e_i$ and $d$ by $e_i^*$ and $d^*$, respectively. Note that in $H^*(N_g)$, we have that $e_i^*e_j^*=0$ for each $i\ne j$ and $(e_i^*)^2=d^*$,~\cite{Ha}.

The following statement, which describes the $\Z_2$-cohomology ring $H^*(SP^n(N_g))$, can be seen as an analog of Theorem~\ref{mc} in the non-orientable setting.

\begin{thm}[\protect{\cite{KS}}]\label{mc2}
    The $\Z_2$-cohomology ring $H^*(SP^n(N_g))$ is generated by the cohomology classes $e_1^*,e_2^*,\ldots,e_g^*$ and $d^*$ under the following relations:
    \vspace{-1.5mm}
    \begin{enumerate}
    \itemsep -0.30em
        \item For each $1\le i\le g$, $(e_i^*)^2=d^*$, and
        \item $e_{i_1}^*\cdots e_{i_r}^*(d^*)^s=0$ whenever $r+s\ge n+1$ for distinct indices $i_1,\ldots,i_r$.
    \end{enumerate}
\end{thm}

We note that it follows from~\cite[Lemma 19]{KS} that $(d^*)^n \ne 0$.

\subsection{LS-category and topological complexity}\label{lstcsect}

We are now in a position to compute the Lusternik--Schnirelmann category of $SP^n(N_g)$ for each $n$ and $g$.

\begin{proof}[Proof of Theorem~\ref{bb}]
    Of course, for each $n,g\ge 1$, we have the dimensional upper bound $\cat(SP^n(N_g))\le\dim(SP^n(N_g))=2n$ due to Theorem~\ref{clbound}. We now obtain the lower bounds separately in the cases $n\le g$ and $n>g$.
    
    First, let $n\le g$. Then we have from Theorem~\ref{mc2} that
    \[
    \left(e_1^*\right)^2 \left(e_2^*\right)^2 \cdots \left(e_n^*\right)^2 = \left(d^*\right)^n \ne 0.
    \]
Therefore, Theorem~\ref{clbound} gives the inequality $2n\le \cu_{\Z_2}(SP^n(N_g))\le\cat(SP^n(N_g))$ when $n\le g$.

Next, let $n>g$. Then we again get from Theorem~\ref{mc2} that
\[
    \left(e_1^*\right)^{2n-2g+2} \left(e_2^*\right)^2 \left(e_3^*\right)^2 \cdots \left(e_g^*\right)^2 = \left(d^*\right)^{n-g+1}\left(d^*\right)^{g-1} = \left(d^*\right)^n \ne 0.
\]
As before, we get $2n\le \cu_{\Z_2}(SP^n(N_g))\le\cat(SP^n(N_g))$ when $n> g$. Hence, we have that $\cat(SP^n(N_g))=2n$ for all $n,g\ge 1$.
\end{proof}

Unlike $\dcat(SP^n(M_g))$, here we can't determine $\dcat(SP^n(N_g))$ when $n\ge 2$ because $\cu_{\Z_2}(X)$ is not a lower bound to $\dcat(X)$. The methods of~\cite[Section 3]{J2} also do not help here because $\pi_1(SP^n(N_g))=\Z^{g-1}\oplus\Z_2$ has non-trivial torsion.

\begin{remark}
    It is well-known that $SP^n(N_1)=\R P^{2n}$, see, for example,~\cite{KS}. We observe that Theorem~\ref{bb} subsumes the LS-category computation for $\R P^{2n}$. Moreover, Theorem~\ref{ee} and Corollary~\ref{ee1} also recover the corresponding results for $\R P^{2n}$ and $\C P^n$ (\emph{cf}. Remark~\ref{obvobv}).
\end{remark}

We now study the LS-category of finite products of the symmetric products of non-orientable surfaces.

\begin{proof}[Proof of Corollary~\ref{bb2}]
The proof is similar to that of Corollary~\ref{aa2}. The proof of Theorem~\ref{bb} implies that $\cu_{\Z_2}(SP^n(N_g))=\cat(SP^n(N_g))$ for each $n$ and $g$. The K\"unneth formula gives the following isomorphism of $\Z_2$-algebras:
\[
H^*(SP^{n_1}(N_{g_1})\times SP^{n_2}(N_{g_2}))=H^*(SP^{n_1}(N_{g_1}))\otimes_{\Z_2} H^*(SP^{n_2}(N_{g_2})).
\]
So, $\cu_{\Z_2}(SP^{n_1}(N_{g_1}))+\cu_{\Z_2}(SP^{n_2}(N_{g_2}))\le \cu_{\Z_2}(SP^{n_1}(N_{g_1})\times SP^{n_2}(N_{g_2}))$, which gives $\cat(SP^{n_1}(N_{g_1}))+\cat(SP^{n_2}(N_{g_2})) = \cat(SP^{n_1}(N_{g_1})\times SP^{n_2}(N_{g_2}))$ in view of~\eqref{summation}. Thus, our desired conclusions follow from an inductive argument.
\end{proof}

We follow Gromov~\cite{Gr} and say that a closed $k$-manifold $M$ is {\em essential} if a map $u:M\to B\pi_1(M)$ that induces an isomorphism of the fundamental groups cannot be deformed to the $(k-1)$-skeleton of the CW complex $B\pi_1(M)$. In particular, every closed aspherical manifold is essential, although the converse is not true.

\begin{proof}[Proof of Corollary~\ref{bb1}]
Recall that a closed $k$-manifold $M$ is essential if and only if $\cat(M)=\dim(M)=k$, see~\cite[Theorem 4.1]{KR} (and also~\cite[Theorem 2.51]{CLOT}). Therefore, given $n\ge 1$, $SP^n(M_g)$ is essential precisely when $n\le g$ because of Theorem~\ref{oldknown}, and $SP^n(N_g)$ is essential for all $g\ge 1$ due to Theorem~\ref{bb}.
\end{proof}

For the $m$-th sequential topological complexity of $SP^n(N_g)$ for each $n,g\ge 1$ and $m\ge 2$, we have the bounds $2n(m-1)\le\TC_m(SP^n(N_g))\le 2nm$ due to Theorem~\ref{tcbound} and Corollary~\ref{bb2}. Let us discuss the case $n=1$ first. For $m=2$, the lower bound is not attained for any $g$, while the upper bound is attained for each $g\ge 2$, see~\cite{Dr2} and~\cite{CV}. Notably, for $m\ge 3$, the upper bound is always attained due to~\cite{GGGL}.

We now focus on the case $n\ge 2$ for $m=2$ and explain why the technique of Section~\ref{tcm} and~\cite[Section 6.C]{DDJ} that helps determine $\TC(SP^n(M_g))$ does not help in determining $\TC(SP^n(N_g))$ for general $n\ge 2$ and $g\ge 1$.

Given the $\Z_2$-cohomology classes $e_i^*\in H^1(SP^n(N_g))$ for integers $1\le i \le g$, let us define the $\Z_2$-special zero-divisors $\ov{e_i^*}=e_i^*\otimes 1+1\otimes e_i^*$. It is easy to see using Theorem~\ref{mc2} that in mod $2$, for any $k\le \min\{n,g\}$, we have that
\[
\left(\ov{e_1^*}\right)^2\left(\ov{e_2^*}\right)^2\cdots\left(\ov{e_k^*}\right)^2=\left(\left(e_1^*\right)^2\otimes 1+1\otimes \left(e_1^*\right)^2\right)\cdots \left(\left(e_k^*\right)^2\otimes 1+1\otimes \left(e_k^*\right)^2\right)
\]
\[
= \left(d^*\otimes 1+1\otimes d^*\right)^k\in H^{2k}(SP^n(N_g)\times SP^n(N_g)).
\]

For any $j\le k$, since $(d^*)^{j} \ne 0$ by~\cite[Lemma 19]{KS}, $(d^*)^{j}$ is the identity of $H^{2j}(SP^n(N_g);\Z_2)$ because we are working in the coefficient field $\Z_2$. Since $H^*(SP^n(N_g);\Z_2)$ is finitely generated~\cite{KS}, the Kunneth isomorphism (see~\cite[Section 3.2]{Ha}) implies that $(d^*)^{k-i}\otimes (d^*)^i=((d^*)^{k-i}\otimes 1) \smile (1\otimes (d^*)^i)$ is the identity of $H^{2k}(SP^n(N_g)\times SP^n(N_g);\Z_2)$ for $0\le i \le k$. In particular, $(d^*)^{k-i}\otimes (d^*)^i$ and $(d^*)^{i}\otimes (d^*)^{k-i}$ both represent the identity of $H^{2k}(SP^n(N_g)\times SP^n(N_g);\Z_2)$. 

Now we notice the following phenomenon.

\begin{prop}\label{lucas}
In $\Z_2$-coefficients, for each $k\ge 1$ the term $(d^*\otimes 1+1\otimes d^*)^k$ has an even number of terms of the form $(d^*)^{k-i}\otimes (d^*)^i$, where $i \in \{0,\ldots, k\}$.   
\end{prop}
\begin{proof}
By the binomial theorem, it suffices to show that ${k\choose i}\equiv 1(\text{mod }2)$ holds for an even number of $i\in\{0,\ldots,k\}$. Clearly, ${k\choose i}$ is odd if and only if ${k\choose k-i}$ is odd. Moreover, when $k$ is even, ${k\choose k/2}$ is also even. To see this, let us write the binary expansions of $k$ and a fixed $i\le k$ as $k=k_0+k_1\cdot2+\cdots+k_{t-1}\cdot 2^{t-1}+k_t\cdot 2^t$ and $i=i_0+i_1\cdot2+\cdots+i_{t-1}\cdot 2^{t-1}+i_t\cdot 2^t$, respectively, where $t\ge 0$ is an integer and $k_j,i_j\in\{0,1\}$ for $0\le j\le t$. Lucas's Theorem from 1878 (see~\cite{Fi} for a short proof) says that
\[
{k\choose i}\equiv\prod_{j=1}^t{k_j \choose i_j}(\text{mod }2).
\]
Hence, ${k\choose i}$ is odd if and only if $i_j=1 \implies k_j=1$ for any $j$, so it follows that ${k\choose k/2}$ is even. Thus, integers $i\in\{0,\ldots,k\}$ for which ${k\choose i}\equiv 1(\text{mod }2)$ holds always exist in pairs. This completes the proof.
\end{proof}
    
Therefore, for $k\le\min\{n,g\}$, we deduce from Proposition~\ref{lucas} that    
\[
\left(\ov{e_1^*}\right)^2\left(\ov{e_2^*}\right)^2\cdots\left(\ov{e_k^*}\right)^2=\left(d^*\otimes 1+1\otimes d^*\right)^k=0
\]
in $H^{2k}(SP^n(N_g)\times SP^n(N_g);\Z_2)$. Hence, the $\Z_2$-zero-divisor cup-length of $SP^n(N_g)$ for $n,g\ge 1$ does not give us a lower bound to $\TC(SP^n(N_g))$ greater than $2n$ if we use techniques from Section~\ref{tcm}. We also don't know whether an upper bound to $\TC(SP^n(N_g))$ smaller than $4n$ can be obtained in general! 
Therefore, the topological complexity of $SP^n(N_g)$ remains undetermined in the cases $n\ge 2$.

\begin{remark}
Since $SP^n(N_1)=SP^n(\R P^2)=\R P^{2n}$ for each $n\ge 1$, any general result on the topological complexity of $SP^n(N_g)$ for $n\ge 2$ must subsume the main result of~\cite{FTY}, which says that $\TC(\R P^{2n})$ is equal to the immersion dimension of $\R P^{2n}$. This indicates that determining $\TC(SP^n(N_g))$ can be very difficult. On the other hand, while we currently don't know how we can improve the bounds $2n(m-1)\le\TC_m(SP^n(N_g))\le 2nm$ for $m\ge 3$ and $n\ge 2$ in general, it is possible that the problem of determining $\TC_m(SP^n(N_g))$ is slightly more approachable, especially in view of the computations of $\TC_m(N_g)$ for $m\ge 3$ from~\cite{GGGL}\hspace{0.2mm}\footnote{\hspace{0.2mm}The author thanks Jesús González for this observation.}.  
\end{remark}

\section*{Acknowledgement}
The author would like to thank Alexander Dranishnikov for suggesting to him the problem of determining the sequential topological complexities of symmetric products of orientable surfaces. The author is very grateful to the two anonymous referees for their valuable suggestions and for recommending various important changes and fixes that helped enhance this paper.

\end{document}